\documentclass{amsart}

\usepackage{amsmath,amsthm,amsfonts,amssymb,bm,graphicx}

\usepackage[backref=page]{hyperref}
\hypersetup{colorlinks=true,citecolor=blue,linkcolor=blue,urlcolor=blue,
pdfstartview=FitH, pdfauthor=Valentin Blomer and Gergely
Harcos,pdftitle=A hybrid asymptotic formula for the second moment of Rankin-Selberg $L$-functions}

\theoremstyle{plain}
\newtheorem{theorem}{Theorem}

\newtheorem{lemma}{Lemma}
\newtheorem{prop}{Proposition}

\theoremstyle{definition}
\newtheorem*{Acknowledgements}{Acknowledgements}
\newtheorem{remark}{Remark}
\newtheorem*{notation}{Notation and conventions}

\renewcommand{\geq}{\geqslant}
\renewcommand{\leq}{\leqslant}
\renewcommand{\mod}{\mathrm{mod}\,}

\DeclareMathOperator{\GL}{GL}
\DeclareMathOperator{\SL}{SL}

\newcommand{\eps}{\varepsilon}
\newcommand{\RR}{\mathbb{R}}
\newcommand{\CC}{\mathbb{C}}
\newcommand{\QQ}{\mathbb{Q}}
\newcommand{\ZZ}{\mathbb{Z}}
\newcommand{\NN}{\mathbb{N}}
\newcommand{\asterisk}{\ \sideset{}{^*}}
\newcommand{\gconst}{\tilde g_\text{const}}
\newcommand{\gser}{\tilde g_\text{ser}}
\newcommand{\ov}[1]{\overline{#1}}

\makeatletter
\DeclareRobustCommand\widecheck[1]{{\mathpalette\@widecheck{#1}}}
\def\@widecheck#1#2{%
    \setbox\z@\hbox{\m@th$#1#2$}%
    \setbox\tw@\hbox{\m@th$#1%
       \widehat{%
          \vrule\@width\z@\@height\ht\z@
          \vrule\@height\z@\@width\wd\z@}$}%
    \dp\tw@-\ht\z@
    \@tempdima\ht\z@ \advance\@tempdima2\ht\tw@ \divide\@tempdima\thr@@
    \setbox\tw@\hbox{%
       \raise\@tempdima\hbox{\scalebox{1}[-1]{\lower\@tempdima\box
\tw@}}}%
    {\ooalign{\box\tw@ \cr \box\z@}}}
\makeatother

\begin{document}

\author{Valentin Blomer}
\address{Mathematisches Institut, Bunsenstr. 3-5, 37073 G\"ottingen, Germany} \email{blomer@uni-math.gwdg.de}

\author{Gergely Harcos}
\address{Alfr\'ed R\'enyi Institute of Mathematics, Hungarian Academy of Sciences, POB 127, Budapest H-1364, Hungary} \email{gharcos@renyi.hu}
\title[An asymptotic formula]{A hybrid asymptotic formula for the second moment of Rankin--Selberg $L$-functions}

\thanks{The first author was supported   by a Volkswagen Lichtenberg Fellowship and a European Research Council (ERC) Starting Grant 258713. The second author was supported by European Community grant ERG~239277 within the 7th Framework Programme and by OTKA grants K~72731 and PD~75126.}


\begin{abstract} Let $g$ be a fixed modular form of full level, and let $\{f_{j, k}\}$ be a basis of holomorphic cuspidal newforms of even weight $k$, fixed level and fixed primitive nebentypus. We consider the Rankin--Selberg $L$-functions $L(1/2 + it, f_{j, k} \otimes g)$ and compute their second moment over $t \asymp T$ and $k \asymp K$.
For $K^{3/4+\eps} \leq T \leq K^{5/4-\eps}$ we obtain an asymptotic formula with a power saving error term.
Our result covers the second moment of $L(1/2 + it+ir, f_{j, k})L(1/2 + it-ir, f_{j, k})$ for any fixed real number $r$, hence
also the fourth moment of $L(1/2 + it, f_{j, k})$.
For the proof we develop a precise uniform approximate functional equation with explicit dependence on the archimedean parameters.
\end{abstract}


\setcounter{tocdepth}{2} \maketitle 

\section{Introduction}

Moments of families of $L$-functions on the critical line play an important role in arithmetic. Often they constitute a crucial input for subconvexity or non-vanishing results that have far-reaching applications. In this article we consider degree four Rankin-Selberg convolutions of a fixed modular form with holomorphic cusp forms of large weight $k$ and their associated $L$-functions at a point $1/2 + it$ high on the critical line. We aim at an asymptotic formula for the second moment with a power saving error term, where $k$ and $t$ are both large, but possibly in different ranges. A special case treats the fourth moment of $L$-functions associated with holomorphic cusp forms.

More precisely, we shall consider the following situation. Let $N\geq 1$ be an integer and $\chi$ an even primitive Dirichlet character modulo $N$, in particular $N\not\equiv 2\pmod{4}$. For even $k\geq 2$ let
\begin{equation}\label{holobasis}
  f_{j, k}(z) = \sum_{m=1}^{\infty} \lambda_{j, k}(m)m^{\frac{k-1}{2}}e(mz), \qquad 1\leq j\leq \theta_{k}(N, \chi) := \text{dim}_{\CC}S_k(N, \chi),
\end{equation}
be an orthogonal basis, consisting of arithmetically normalized newforms, for the cuspidal space $S_{k}(N, \chi)$; such a basis exists by the assumption that $\chi$ is primitive. It follows from \cite[Th\'eor\`{e}me~1]{CH} that
with the exception of finitely many pairs $(N,k)$ we have (see the end of this introduction for a precise statement)
\begin{equation}\label{dim}
  \theta_k(N, \chi) = kN^{1+o(1)}.
\end{equation}
For a cusp form $f \in S_k(N, \chi)$ we define
\begin{displaymath}
 \rho(f) := \frac{\Gamma(k-1)}{(4\pi)^{k-1} \| f\|^2},
\end{displaymath}
where
\[\|f\|^2 := \int_{\Gamma_0(N) \backslash \mathbb{H}} |f(z)|^2 y^{k-2} \,dxdy\] is the usual Petersson norm. For  the Hecke eigenform $f_{j, k}$ it is known by the work of Iwaniec and Hoffstein--Lockhart~\cite{Iw2,HL} that
\begin{equation}\label{defrho}
  \rho(f_{j, k}) = (kN)^{-1+o(1)}.
\end{equation}
Let $g$ be a fixed automorphic form for the full modular group, which can be
\begin{itemize}
\item a holomorphic cusp form of some weight $k' < k$ with Hecke eigenvalues $\lambda_g(n)$;
\item a Maa{\ss } cusp form of weight zero, spectral parameter $r = \sqrt{\lambda-1/4} \geq 0$, sign $\epsilon_g \in \{\pm 1\}$, with Hecke eigenvalues $\lambda_g(n)$;
\item an Eisenstein series $E_{r} := E(z, 1/2+ir)$ with $r \in \RR \setminus \{0\}$, or  $E_0 := \frac{\partial}{\partial s} E(z, s) |_{s = 1/2}$, with Hecke eigenvalues $\lambda_g(n) =  \sum_{ab=n} (a/b)^{ir}$.
\end{itemize}
Let
\begin{equation}\label{RS}
  L(s, f_{j, k} \otimes g) = L(2s, \chi) \sum_{n=1}^{\infty} \frac{\lambda_{j, k}(n) \lambda_g(n)}{n^s}
\end{equation}
be the corresponding Rankin--Selberg $L$-function. Our assumptions on $f_{j, k}$ and $g$ imply that for $g$ cuspidal this is really the usual Rankin--Selberg $L$-function (see \cite{Li}, in particular Example~2 on p.~146).
This $L$-function is entire (and by a deep result of Ramakrishnan~\cite{Ra2} it is even cuspidal on $\GL_4$ over $\QQ$, but we will not need this). In the case  $g = E_r$ we take the right hand side of \eqref{RS} as the definition of the left hand side, then by the multiplicative properties of Hecke eigenvalues we obtain
\begin{equation}\label{Eistwist}
  L(s, f_{j, k} \otimes E_{r}) = L(s+ir, f_{j, k})L(s-ir, f_{j, k}).
\end{equation}

Let $W_1, W_2:(0,\infty)\to[0,\infty)$ be fixed smooth functions with nonempty support in $[1, 2]$, and let $T, K \geq 1$ be two sufficiently large parameters. The aim of this article is to obtain an asymptotic formula for
\begin{equation}\label{defI}
\mathcal{I}(T, K) := \int_{0}^{\infty} W_1\left(\frac{t}{T}\right)\sum_{k \equiv 0 \, (2)} W_2\left(\frac{k-1}{K}\right)  \sum_{j=1}^{\theta_{k}(N, \chi)}\rho(f_{j, k}) |L(1/2 + it, f_{j, k} \otimes g)|^2 \,dt.
\end{equation}
The square root of the analytic conductor of $|L(1/2 + it, f_{j, k} \otimes g)|^2$ is approximately given by (cf.\ \eqref{para3} below)
\begin{equation}\label{simplecond}
  \mathcal{C}(t, k) := \frac{N^2}{(2\pi)^4}\left(t^2+\frac{k^2}{4}\right)^2.
\end{equation}
For $j \in \NN_0$ and $r \in \RR$ we shall need the following smooth averages:
\begin{equation}\label{smoothlog}
\begin{split}
\mathcal{L}_j(T, K) &  := \frac{1}{TK}
\int_{0}^{\infty} \int_{0}^{\infty}   W_1\left(\frac{t}{T}\right)  W_2\left(\frac{x}{K}\right)  \log^{j}\mathcal{C}(t, x) \,dt\, dx\asymp \log^j(T+K), \\
\mathcal{M}_{ir}(T, K) & := \frac{1}{TK}
\int_{0}^{\infty} \int_{0}^{\infty}   W_1\left(\frac{t}{T}\right)  W_2\left(\frac{x}{K}\right)   \mathcal{C}(t, x)^{ir} \,dt\, dx \ll 1.
\end{split}
\end{equation}

The shape of the asymptotic formula depends on the type of $g$, hence we formulate our main results in three separate theorems.  We derive an asymptotic formula which is nontrivial precisely when
\begin{equation}\label{assumption}
  K^{3/4+\eps} \leq T \leq K^{5/4-\eps}
\end{equation}
for some $\eps>0$.

\begin{theorem}\label{theorem1} For $g$ cuspidal there are constants $a_0,a_1\in\RR$ depending only on $N$ and $g$ such that
\begin{displaymath}
  \mathcal{I}(T, K) = TK \bigl(a_1 \mathcal{L}_1(T, K) + a_0 \mathcal{L}_0(T, K)\bigr)+O_{N, g, W_1, W_2, \eps}\left((TK)^{1+\eps}(T^4K^{-5}+T^{-3}K^4)\right).
\end{displaymath}
The leading constant is given by
\begin{displaymath}
  a_1 =  \frac{1}{2}L(1, {\rm Ad}^2 g)\prod_{p \mid N} \left(1-p^{-2}\right).
\end{displaymath}
\end{theorem}

\begin{theorem}\label{theorem2} For $g=E_r$ with $r\neq 0$ there are constants $a_0,a_1,a_2\in\RR$ and $b_\pm\in\CC$ depending only on $N$ and $r$ such that
\begin{displaymath}
\begin{split}
  \mathcal{I}(T, K) = & TK \biggl(\sum_{j=0}^2 a_j \mathcal{L}_j(T, K) +
  b_+ \mathcal{M}_{ir}(T, K)+ b_-\mathcal{M}_{-ir}(T, K)\biggr) \\
  & + O_{N, r, W_1, W_2, \eps} \left((TK)^{1+\eps}(T^4K^{-5}+T^{-3}K^4)\right).
  \end{split}
\end{displaymath}
The leading constant is given by
\[a_2=\frac{1}{8}|\zeta(1+2ir)|^2\prod_{p \mid N} \left(1-p^{-2}\right);\]
moreover,
\begin{displaymath}
  b_{\pm} = \frac{1}{2} \zeta(1\pm 2ir)^4 \prod_{p\mid N}(1-p^{-2\mp 4ir}).
\end{displaymath}
\end{theorem}

\begin{theorem}\label{theorem3} For $g=E_0$ there are constants $a_0,a_1,a_2,a_3,a_4\in\RR$ depending only on $N$  such that
\begin{displaymath}
  \mathcal{I}(T, K) = TK \biggl(\sum_{j=0}^4 a_j \mathcal{L}_j(T, K)  \biggr)+ O_{N, W_1, W_2, \eps} \left((TK)^{1+\eps}(T^4K^{-5}+T^{-3}K^4)\right).
\end{displaymath}
The leading constant is given by
\[a_4=\frac{1}{384}\prod_{p \mid N} \left(1-p^{-2}\right).\]
\end{theorem}

\begin{remark}
For $T=K$ and fixed $W_{1,2}$ the above asymptotic formulae take a particularly simple shape as
$\mathcal{L}_j(K, K)$ is a polynomial in $\log K$ of degree $j$ and $\mathcal{M}_{ir}(K,K)$ is proportional to $K^{4ir}$.
\end{remark}

The present paper was inspired by recent work of Kim--Zhang~\cite{KZ} based on a similar result of Duke~\cite{Du} which in turn was motivated by a paper of Sarnak~\cite{Sa1}. These three papers have in common that they estimate an archimedean family of $L$-functions of conductor $T^8$ and size $T^2$ \cite{Sa1} (resp.\ $T^3$ \cite{Du,KZ}), without using approximate functional equations and orthogonality properties of the family. Rather, the argument is based on $L^2$-techniques, and the idea is to express the moment as the norm of an automorphic form on hyperbolic 3-space \cite{Sa1} (resp.\ 5-space \cite{Du}) when restricted to a certain cone. This very elegant and flexible approach gives upper bounds with the correct log-power, but (in its basic form) does not yield an asymptotic formula of the above kind. Moreover, it requires the various archimedean parameters to be of the same size. With more work, these problems could be dealt with essentially by treating the off-diagonal contribution nontrivially as in \cite{PS}.
We also note that there are many other results in the literature on various moments of Rankin--Selberg $L$-functions, e.g.\ \cite{HM, LLY, Sa}. Particularly interesting is the preprint \cite{Y}, where an upper bound for the    $\GL_3 \times \GL_2$ analogue of the quantity \eqref{defI} for $T \approx K$ is obtained. The paper \cite{Kh} is also somewhat similar in spirit to the present work. General conjectures for a wide variety of moments of $L$-functions have been formulated in \cite{C+}, see in particular Conjecture~4.5.1, where the family of Theorems~\ref{theorem2} and \ref{theorem3} is considered for fixed $t$ and $k$. It seems very complicated to integrate the right hand side of \cite[(4.5.8)]{C+} over $t$ and $k$.\\

We use essentially the same type of family as in \cite{Du,KZ}. The novelty in our work is the derivation of an asymptotic formula with independent parameters $T$ and $K$; there are not many results of this type in the literature. An upper bound of essentially the right order of magnitude for $\mathcal{I}(T,K)$ in Theorem~\ref{theorem1} follows immediately from \cite[Section~3]{JM}, even without integration over $t$. The main point here is really to prove an asymptotic formula with a power saving error term which turns out to be quite strong in the case $T=K$. Our family has conductor $(T+K)^8$ and size $TK^2$, hence in the extreme cases of Theorem~\ref{theorem1} (cf.\ \eqref{assumption}) we are dealing with a family of conductor $K^8$ and size $K^{11/4}$, or a family of conductor $T^8$ and size $T^{13/5}$.
Theorem~\ref{theorem3} states a precise asymptotic formula for the fourth moment of certain modular $L$-functions (cf.\ \eqref{Eistwist}--\eqref{defI}), improving substantially on \cite[Theorem~1]{Du} and \cite[Theorem~2]{KZ}.

Our approach is based on classical techniques, in particular an approximate functional equation, Petersson's formula, Voronoi summation etc. It may be observed, however, that at least in the range \eqref{assumption}  we do not have to treat shifted convolutions sums explicitly, although of course they appear throughout the discussion of the off-diagonal term. In addition, we do not have to use any nontrivial bound for twisted
Kloosterman sums.

As an interesting feature we remark that in the Eisenstein case the off-diagonal term contributes to the main terms in Theorems~\ref{theorem2} and \ref{theorem3}, and it requires a non-trivial manipulation to recover the exact shape of this off-diagonal contribution. The shape of the asymptotic formula in Theorem \ref{theorem2} is particularly interesting, as it contains oscillating secondary terms $\mathcal{M}_{\pm ir}(T,K)$. This seems to be a new phenomenon of this family. The fact that the off-diagonal term contributes towards the main term seems to be a general feature in moment computations of Rankin-Selberg $L$-functions with Eisenstein series, see for instance \cite{Bl} or \cite{DFI}. It would be desirable to have a conceptual understanding of this empirical phenomenon. It also indicates that there is some intrinsic difficulty in deriving an asymptotic formula for a fourth moment of automorphic $L$-functions as in Theorems~\ref{theorem2} and \ref{theorem3}.

Our method could in principle treat more general assumptions on the parameters of $f_{j, k}$ and $g$, but this would infer technicalities that may obscure the paper rather than be useful in applications. It is  possible by standard techniques to remove the harmonic weight $\rho(f_{j, k})$, see e.g.\ \cite[Section~26.5]{IK}. It should also be possible to replace the family of holomorphic cusp forms of weights $k \asymp K$ by a basis of Maa{\ss} forms of spectral parameters $t_j\asymp K$. In this situation, however, the analysis would be somewhat different, and in certain ranges one would encounter effects where the conductor drops, see e.g.\ \cite{Y1} for a related discussion. Finally and most importantly, it is possible to prove asymptotic formulae in the spirit of Theorem~\ref{theorem1} in other ranges, for example in $K^{\varepsilon} \leq T \leq K$, and most likely also for $T$ considerably bigger than $K^{5/4}$. This would require, however, a careful treatment of shifted convolution sums, in the latter case of large $T$ with oscillating weight functions. As remarked above, the present approach avoids shifted convolution sums.

In the course of the proof we use two auxiliary results that may be useful in other situations. These are certainly known to specialists, but we could not find them in the literature. Proposition~\ref{prop1} proves a relatively sophisticated form of the approximate functional equation, while Proposition~\ref{voro} establishes a Voronoi summation formula for Eisenstein series $E_r$.

Finally we remark that \cite[Th\'eor\`{e}me~1]{CH} can be used to show that
$\theta_{k}(N, \chi)=0$ for at least one even primitive $\chi$ if and only if the pair $(N,k)$ is one of
$(1,2)$, $(1,4)$, $(1,6)$, $(1,8)$, $(1,10)$, $(1,14)$, $(5,2)$, $(7,2)$, $(8,2)$, $(9,2)$, $(11,2)$, $(12,2)$,
$(13,2)$, $(15,2)$, $(17,2)$, $(19,2)$, $(21,2)$. In other words, \eqref{dim} holds for $(N,k)$ outside this explicit set.

\begin{notation} In an effort to lighten the notational burden, we will use the following conventions. The symbol $\eps$ denotes an arbitrarily small positive constant and the letter $A$ denotes an arbitrarily large positive constant, not necessarily the same at each occurrence. All implied constants may depend on $\eps$, $A$, $W_1$, $W_2$. This allows us to write, for instance, $C^{\eps} \log C \ll C^{\eps}$. All implied constants may also depend on $N$ and the parameters of the fixed modular form $g$; the reader may check that the dependence on these quantities is always polynomial. The phrase ``negligible error" means an error of size $O\bigl((T+K)^{-A}\bigr)$. By \eqref{assumption} this is the same as $O(T^{-A})$ or $O(K^{-A})$.
\end{notation}

\begin{Acknowledgements} We thank the Institute for Advanced Study for inviting us to the inspiring ``Workshop on Analytic Number Theory" in March 2010, where good progress was made on this paper. We also thank the referee for a careful reading and numerous comments.
\end{Acknowledgements}

\section{Preparatory material}

\subsection{A uniform approximate functional equation}

We start with a general result that may be of independent interest. In order to prove an asymptotic formula for $\mathcal{I}(T, K)$, it is very important to have an approximate functional equation with a weight function that is independent of the archimedean parameters. An approximate functional equation of this kind was developed in general in \cite[Theorem~2.5]{Ha}.  Here we need a slightly stronger variant of this result with a better error term, which forces us to use a somewhat more complicated main term, similarly as in \cite[Lemma~1]{HB}. We shall prove a general result since it requires the same amount of work as the proof of the special case of Rankin--Selberg $L$-functions we are interested in here.

At this point we take the opportunity to correct an inaccuracy in our earlier work \cite{BH}. A uniform approximate functional equation as stated in Proposition~\ref{prop1} below should have been used in \cite[(2.12)]{BH}. This would justify, at the cost of an admissible error $O_{\eps,A}\left(D^{1/2+\eps}T^{-A}\right)$, the tacit assumption that the weight function $V$ in \cite[(7.2)]{BH} is essentially independent of $t$.\\

We keep the notation of \cite{Ha}. Let $F$ be a number field of degree $d$, and let $\pi = \otimes_v\pi_v$  be an isobaric automorphic representation of $\GL_m$ over $F$ (cf.\ \cite{Ra}) with unitary central character and contragradient representation $\tilde{\pi}$. The corresponding $L$-functions are defined for $\Re s > 1$  by absolutely convergent Dirichlet series as
\begin{displaymath}
  L(s, \pi) = \sum_{n=1}^{\infty} \frac{a_n}{n^s}, \qquad L(s, \tilde{\pi}) = \sum_{n=1}^{\infty} \frac{\ov{a_n}}{n^s}
\end{displaymath}
which extend to meromorphic functions on $\CC$ with finitely many poles,
and these are connected by a functional equation of the form
\begin{displaymath}
  \mathcal{N}^{s/2} L(s, \pi_{\infty}) L(s, \pi) = \kappa \mathcal{N}^{(1-s)/2} L(1-s, \tilde{\pi}_{\infty})L(1-s, \tilde{\pi}).
\end{displaymath}
Here $\mathcal{N}$ is the conductor (a positive integer), $\kappa$ is the root number (of modulus 1) and
\begin{displaymath}
  L(s, \pi_{\infty}) = \prod_{j=1}^{md} \pi^{-\frac{s+\mu_j}{2}} \Gamma\left(\frac{s+\mu_j}{2}\right), \qquad L(s, \tilde{\pi}_{\infty}) = \prod_{j=1}^{md} \pi^{-\frac{s+\ov{\mu_j}}{2}} \Gamma\left(\frac{s+\ov{\mu_j}}{2}\right)
\end{displaymath}
for certain $\mu_j \in \CC$ which satisfy
\begin{equation}\label{LRS}
\Re \mu_j  \geq  \frac{1}{m^2+1} - \frac{1}{2}
\end{equation}
by a result of Luo--Rudnick--Sarnak~\cite{LRS}. We put
\begin{displaymath}
\eta_j :=  \frac{1}{4}+\frac{\mu_j}{2},\qquad
  \eta := \min_{1 \leq j \leq md} |\eta_j|, \qquad \lambda := \frac{L(1/2, \tilde{\pi}_{\infty})}{L(1/2, \pi_{\infty})},
\end{displaymath}
and we define the analytic conductor (at $s=1/2$) as
\begin{displaymath}
  C := \frac{\mathcal{N}}{(2\pi)^{md}} \prod_{j=1}^{md} \left|\frac{1}{2}+\mu_j\right| = \frac{\mathcal{N}}{\pi^{md}} \prod_{j=1}^{md} |\eta_j|.
\end{displaymath}
For a multi-index $\bm n \in \NN_0^{2md}$ we write $|\bm n|: = n(1) + \cdots + n(2md)$ and
\begin{equation}\label{fateta}
  {\bm\eta}^{-\bm n} := \prod_{j=1}^{md} \eta_j^{-{\bm n}(2j-1)} \ov{\eta_j}^{\,-{\bm n}(2j)}.
\end{equation}
By a result of Molteni~\cite[Theorem~4]{Mo}, the coefficients of $L(s,\pi)$ satisfy the uniform bound
\begin{equation}\label{molteni}
\sum_{n\leq x}|a_n|\ll_\eps x^{1+\eps}C^\eps.
\end{equation}
To be precise, axiom (A4) in \cite{Mo} includes $\Re\mu_j\geq 0$ but \eqref{LRS} is sufficient for the proof.

Let us assume that $L(s,\pi)$ is entire, then we can state the following result:

\begin{prop}\label{prop1} Let $G_0 : (0, \infty) \rightarrow \RR$ be a smooth function with functional equation $G_0(x) + G_0(1/x) = 1$ and derivatives decaying faster than any negative power of $x$ as $x \rightarrow \infty$. Let $M \in \NN$. There are explicitly computable rational constants $c_{{\bm n},\ell} \in \QQ$ depending only on ${\bm n}$, $\ell$, $M$, $m$, $d$ such that the following holds for
\begin{equation}\label{Gdef}
  G(x)  :=   G_0(x)+\sum_{\substack{0< |{\bm n}| < M \\ 0< \ell < |{\bm n}|+M}} c_{{\bm n},\ell} {\bm\eta}^{-\bm n}   \left(x\frac{\partial}{\partial x}\right)^{\ell} G_0(x).
\end{equation}

a) The function $G$  is smooth and its Mellin transform $\widecheck G$ is holomorphic everywhere except for a simple pole at $s=0$ with formal Laurent expansion
\begin{equation}\label{laurent}
  \widecheck G(s) \sim \frac{1}{s} + \sum_{j=0}^\infty c_j s^j, \qquad  c_j = \frac{-1}{(j+1)!} \int_0^{\infty} G'(x) (\log x)^{j+1} \,dx = d_j + O(\eta^{-1}).
\end{equation}
Moreover, one has
\begin{equation}\label{boundsG}
\frac{  \partial^j}{\partial x^j} G(x) \ll (1+x)^{-A}; \qquad \widecheck G(s) \ll (1+|s|)^{-A}\ \text{for $|\Re s|<\sigma_0$ and $|s| > 1$}.
\end{equation}
Here $d_j$ and the implied constants depend at most on $j$, $A$, $\sigma_0$, $M$, $m$, $d$, and the function $G_0$.
\\

b) For any $\eps > 0$ one has
\begin{equation}\label{approxfuncteq}
  L(1/2, \pi) = \sum_{n=1}^{\infty} \frac{a_n}{\sqrt{n}} G\left(\frac{n}{\sqrt{C}}\right) + \kappa\lambda \ov{\sum_{n=1}^{\infty} \frac{a_n}{\sqrt{n}} G\left(\frac{n}{\sqrt{C}}\right)} + O\bigl(\eta^{-M}C^{1/4+\eps}\bigr),
\end{equation}
where the implied constant depends at most on $\eps$, $M$, $m$, $d$, and the function $G_0$.
\end{prop}

\begin{remark}\label{selfdual} For self-contragradient representations $\pi$ we can strengthen Proposition~\ref{prop1} by imposing the additional symmetry $c_{\ov{\bm n},\ell}=c_{\bm n,\ell}$ for the involution
$\bm n\mapsto\ov{\bm n}$ defined by ${\bm\eta}^{-\ov{\bm n}}={\ov{\bm\eta}}^{\,-\bm n}$, that is $\ov{\bm n}(2j-1):=\bm n(2j)$ and $\ov{\bm n}(2j):=\bm n(2j-1)$. Then $G$ defined by \eqref{Gdef} is real-valued and the two main terms in \eqref{approxfuncteq} are identical. Indeed, in the self-contragradient
situation $\kappa=\lambda=1$ and the coefficients $a_n$ are real, hence we can replace the original $G$ by its real part without affecting the validity of \eqref{laurent}--\eqref{approxfuncteq}. This corresponds to replacing
the original coefficients $c_{\bm n,\ell}$ by $(c_{\bm n,\ell}+c_{\ov{\bm n},\ell})/2$, upon noting that the original coefficients are rational.
\end{remark}

\begin{proof} We follow closely the proof of \cite[Theorem~2.5]{Ha}. We can write $G_0$ as an inverse Mellin transform
\begin{equation}\label{G0def}
G_0(x)=\frac{1}{2\pi i}\int_{(\sigma)}x^{-s}H(s)\frac{ds}{s},\qquad\sigma>0,
\end{equation}
where $H(s):=s\widecheck G_0(s)$ extends to an entire functions satisfying
\begin{equation}\label{Hfeature}
H(0)=1,\qquad H(s)=H(-s)=\ov{H(\ov{s})},\qquad H(s) \ll (1+|s|)^{-A}\ \text{for $|\Re s|<\sigma_0$}.
\end{equation}
As in the Erratum of \cite{Ha} we define
\begin{displaymath}
  F(s, \pi_{\infty}) := \frac{1}{2} C^{-s/2}\mathcal{N}^s \frac{L(1/2+s, \pi_{\infty}) L(1/2, \tilde{\pi}_{\infty})}{L(1/2-s, \tilde{\pi}_{\infty})L(1/2, \pi_{\infty})} + \frac{1}{2} C^{s/2},
\end{displaymath}
then by the proof of \cite[Theorem~2.1]{Ha} we have the exact formula
\[L(1/2, \pi) = \sum_{n=1}^{\infty} \frac{a_n}{\sqrt{n}} W\left(\frac{n}{\sqrt{C}}\right) + \kappa\lambda \ov{\sum_{n=1}^{\infty} \frac{a_n}{\sqrt{n}} W\left(\frac{n}{\sqrt{C}}\right)}\]
with the weight function
\[W(x)=\frac{1}{2\pi i}\int_{(\sigma)}x^{-s}C^{-s/2}F(s,\pi_\infty)H(s)\frac{ds}{s},\qquad\sigma>0.\]
Now the idea is to approximate $W$ by a function $G$ of the form (cf.\ \eqref{fateta})
\[G(x)=\frac{1}{2\pi i}\int_{(\sigma)}x^{-s}
\biggl(1 + \sum_{\substack{0< |{\bm n}| < M \\ 0< \ell < |{\bm n}|+M}} c_{{\bm n},\ell} {\bm\eta}^{-\bm n}(-s)^\ell\biggr)H(s)\frac{ds}{s},\qquad\sigma>0,\]
which, by \eqref{G0def}, is precisely the function defined in \eqref{Gdef}. In fact
\begin{equation}\label{MellinGdef}
\widecheck G(s) = \biggl(1 + \sum_{\substack{0< |{\bm n}| < M \\ 0< \ell < |{\bm n}|+M}} c_{{\bm n},\ell} {\bm\eta}^{-\bm n}(-s)^\ell\biggr)\widecheck G_0(s),
\end{equation}
hence \eqref{laurent} and \eqref{boundsG} are immediate from \eqref{Hfeature} and $\widecheck G_0(s)=H(s)/s$.
The rapid decay of $W$ together with \eqref{molteni} implies \eqref{approxfuncteq} as soon as the constants $c_{{\bm n}, \ell}$ are chosen such that
\[W(x) = G(x) + O_{M,m,d}(\eta^{-M}).\]
We shall derive this from
\[W(x)-G(x)=\frac{1}{2\pi i}\int_{(0)}x^{-s}
\biggl(C^{-s/2}F(s,\pi_\infty)-1 - \sum_{\substack{0< |{\bm n}| < M \\ 0< \ell < |{\bm n}|+M}} c_{{\bm n},\ell} {\bm\eta}^{-\bm n}(-s)^\ell\biggr)H(s)\frac{ds}{s}\]
by showing that for explicitly computable rational constants $c_{{\bm n},\ell} \in \QQ$ we have
\begin{equation}\label{asymp}
  C^{-it/2}F(it, \pi_{\infty}) = 1 + \sum_{\substack{0< |{\bm n}| < M \\ 0< \ell < |{\bm n}|+M}} c_{{\bm n},\ell} {\bm\eta}^{-\bm n}(-it)^\ell + O_{M,m,d}\left( \eta^{-M}\left(|t|+|t|^{2M}\right)\right),\quad t\in\RR.
\end{equation}
Note that this strengthens \cite[Lemma~4.1]{Ha}.

In proving \eqref{asymp} we can restrict to the range $|t|<\eta$,
because for $|t|\geq\eta$ the approximation is trivial for any constants $c_{{\bm n}, \ell}$. Indeed, $\eta\gg_{m}1$ by \eqref{LRS}, hence we have
\[1\ll_{M,m}|t|^{M}\leq\eta^{-M}|t|^{2M},\] and similarly for
$0< |{\bm n}| < M$ and $0< \ell < |{\bm n}| +M$ we have
\[c_{{\bm n}, \ell} {\bm\eta}^{-\bm n}(-it)^\ell\ll_{M,m,d}\eta^{-|{\bm n}|}|t|^\ell\ll_{M,m}\eta^{-|{\bm n}|}|t|^{|\bm n|+M}\leq\eta^{-M}|t|^{2M}.\]
Let us now assume that $|t|<\eta$, then our starting point is the identity
\begin{align*}
C^{-it/2}F(it, \pi_{\infty}) &= \frac{1}{2} + \frac{1}{2}\prod_{j=1}^{md} \left|\frac{1}{4}+\frac{\mu_j}{2}\right|^{-it}\frac{\Gamma(\frac{1}{4}+\frac{\mu_j}{2}+\frac{it}{2})\Gamma(\frac{1}{4} +\frac{\ov{\mu_j}}{2}) }{\Gamma(\frac{1}{4}+\frac{\ov{\mu_j}}{2} -\frac{it}{2})\Gamma(\frac{1}{4} +\frac{\mu_j}{2})}  \\
& = \frac{1}{2} + \frac{1}{2}\exp\left\{2 i \Re \sum_{j=1}^{md}\int_0^{t/2} \left(\frac{\Gamma'}{\Gamma}\left(\eta_j + i\tau\right) - \log\eta_j\right)  d\tau\right\}.
\end{align*}
In the integral we have $|i\tau| < \eta/2$ and $|\eta_j + i\tau| > \eta/2$, hence the well-known asymptotic expansion \begin{displaymath}
  \frac{\Gamma'}{\Gamma}(z) = \log z - \sum_{0< n < M} \frac{B_n}{nz^n} + O_{\sigma, M}(|z|^{-M}), \qquad \Re z \geq \sigma > 0,
\end{displaymath}
where $B_n\in\QQ$ is the $n$-th Bernoulli number (cf.\ \cite[8.361.8]{GR}), yields
\[\int_0^{t/2} \left(\frac{\Gamma'}{\Gamma}\left(\eta_j + i\tau\right) - \log\eta_j\right) d\tau=\sum_{\substack{0<n<M\\1\leq\ell\leq n+1}} b_{n,\ell}\eta_j^{-n}i^{\ell-1}t^\ell+O_M\left(\eta^{-M}\left(|t|+|t|^{M+1}\right)\right)\]
for some constants $b_{n,\ell}\in\QQ$. It follows that
\begin{align*}
2i\Re\sum_{j=1}^{md}\int_0^{t/2} &\left(\frac{\Gamma'}{\Gamma}\left(\eta_j + i\tau\right) - \log\eta_j\right)  d\tau\\
&=\sum_{\substack{0< n<M\\1\leq\ell\leq n+1}} b_{n,\ell}(it)^\ell
\sum_{j=1}^{md}\left(\eta_j^{-n}+(-1)^{\ell-1}\ov{\eta_j}^{\,-n}\right)
+O_{M,m,d}\left(\eta^{-M}\left(|t|+|t|^{M+1}\right)\right).\end{align*}
Finally we approximate uniformly the exponential function {\it for
imaginary arguments} by its Taylor polynomial of degree $M-1$ to arrive at \eqref{asymp} for some
$c_{{\bm n},\ell}\in\QQ$ depending only on ${\bm n}$, $\ell$, $M$, $m$, $d$. Here we use that for $M\leq |{\bm n}|\leq M^2$ and $1\leq\ell\leq |{\bm n}|+M$ we have
\[\eta^{-|{\bm n}|}|t|^\ell\leq\eta^{-|{\bm n}|}\left(|t|+|t|^{|{\bm n}|+M}\right)\ll_{M,m}\eta^{-M}\left(|t|+|t|^{2M}\right),\]
because $\eta\gg_m 1$ by \eqref{LRS} and $|t|<\eta$.
\end{proof}

We shall fix once and for all a weight function $G_0:(0,\infty)\to\RR$ as in Proposition~\ref{prop1}, so we will not
display the dependence of our statements on this function. We shall fix $M\in\NN$ later in the paper. For $t$ real and $k\geq 2$ an integer
we shall apply Proposition~\ref{prop1} to the particular $L$-function
\begin{equation*}
  \mathfrak{L}_{t,j,k}(s) := L(s + it, f_{j, k} \otimes g)\ov{L(\ov{s} + it, f_{j, k} \otimes g)}.
\end{equation*}
This $L$-function is associated with a self-contragradient representation $\pi$, namely
the isobaric sum of $f_{j, k} \otimes g\otimes|\det|^{it}$ and its contragradient. According to
Remark~\ref{selfdual} we can and we shall assume that $G$ is real-valued, then \eqref{approxfuncteq} becomes
\begin{equation}\label{approxfuncteq2}
|L(1/2 + it, f_{j, k} \otimes g)|^2=2\sum_{n=1}^{\infty} \frac{a_n}{\sqrt{n}} G\left(\frac{n}{\sqrt{C}}\right)+O\bigl(\eta^{-M}C^{1/4+\eps}\bigr).
\end{equation}
We now compute all the necessary data occurring in this situation.  We have
\begin{equation}\label{para1}
  m=8, \qquad d=1, \qquad   \mathcal{N} = N^4.
\end{equation}
Using \eqref{RS}, we find that the Dirichlet coefficients of $\mathfrak{L}_{t,j,k}(s)$ are given by
\begin{equation}
  \quad a_n = \sum_{d_1^2d_2^2m_1m_2= n} \chi(d_1)\ov{\chi(d_2)}  d_1^{-2it}d_2^{2it}\lambda_{j, k}(m_1)\ov{\lambda_{j, k}(m_2)}\lambda_g(m_1)\ov{\lambda_g(m_2)}m_1^{-it}m_2^{it},
\end{equation}
so that \eqref{molteni} is satisfied by previous remarks or by standard bounds. There are constants
$\nu_1,\dots,\nu_4\in\CC$ depending only on $g$ such that the archimedean parameters $\mu_1,\dots,\mu_8\in\CC$ of $\mathfrak{L}_{t,j,k}(s)$ are given by
the numbers $\frac{k-1}{2}+\nu_j+it$ and their complex conjugates. Hence we have
\begin{equation}\label{para3}
\begin{gathered}
\eta_j = \frac{k}{4} +\frac{\Re\nu_j}{2}\pm \frac{i(\Im\nu_j+t)}{2} , \qquad \eta \asymp k+|t|,
\\ C = C_{t, k} = \frac{N^4}{(2\pi)^8}\prod_{j=1}^4 \left|\frac{k}{2} + \nu_j+it\right|^2 \asymp (k+|t|)^8.
\end{gathered}
\end{equation}

\subsection{Voronoi summation}

The purpose of this section is to compile summation formulae for the Hecke eigenvalues $\lambda_g(n)$
twisted by a finite order additive character. This generalizes Voronoi's original formula for the
divisor function (without twist).

\begin{prop}\label{voro} Let $a$ and $c$ be coprime positive integers, and let $F:(0,\infty)\to\CC$ be a smooth function of compact support. Then
\begin{equation}\label{vor}
c\sum_{n=1}^\infty \lambda_g(n)e\biggl(n\frac{a}{c}\biggr) F(n) =
\sum_{n=1}^\infty \lambda_{g}(n) \sum_{\pm} e\biggl(\mp n \frac{\ov{a}}{c}\biggr) \int_{0}^{\infty} F(x)\, J_g^{\pm}\left(\frac{4\pi \sqrt{nx}}{c}\right)dx,
\end{equation}
where
\begin{displaymath}
  J_g^{+}(x) := 2\pi i^k J_{k-1}(x), \qquad J_g^{-}(x) := 0
\end{displaymath}
if $g$ is a holomorphic cusp form of level $1$ and weight $k$;
\begin{equation}\label{Jdef}
J_g^{+}(x) := \frac{-\pi}{\cosh(\pi r) }\bigl(Y_{2 i r}(x) + Y_{-2 i r}(x)\bigr), \qquad J_g^-(x) :=  \epsilon_g 4 \cosh(\pi r) K_{2ir}(x)
\end{equation}
if $g$ is a Maa{\ss} cusp form of level $1$, weight $0$, Laplacian eigenvalue $1/4+r^2$, and sign $\epsilon_g \in \{\pm 1\}$. For $g = E_r$ $(r\in\RR)$ the same formula holds with $J_g^{\pm}$ as in the Maa{\ss} case (with $\epsilon_g = 1$), except that on the right-hand side the following polar term has to be added:
\begin{align}
\label{vorpolar}
&&&&&&&\sum_{\pm}\zeta(1\pm 2ir)\int_0^{\infty} \left(\frac{x}{c^2}\right)^{\pm ir}F(x)\,dx& &\text{for $r\neq 0$,}&&&&&&\\[4pt]
\label{vorpolar2} &&&&&&&\int_0^{\infty}\left(\log\left(\frac{x}{c^2}\right) + 2\gamma\right)\,F(x)\,dx& &\text{for $r=0$.}&&&&&&
\end{align}
\end{prop}

\begin{proof} For Maa{\ss} cusp forms the result is due to Meurman~\cite[Theorem~2]{Me}; for holomorphic
cusp forms it is due to Jutila~\cite[Section~1.9]{Ju2} and Duke--Iwaniec~\cite[Theorem~4]{DI}; for $g=E_0$ it is due to Jutila~\cite[Theorem~4]{Ju}; for $g=E_r$ $(r\neq  0)$ we provide the proof below.

Let $g=E_r$ $(r\neq  0)$, then the Hecke eigenvalues can be explicitly described as
\[\lambda_g(n)=\sum_{ab=n}\left(\frac{a}{b}\right)^{ir}.\]
We shall derive the Voronoi formula from analytic properties of the Dirichlet series
\[D(g,x,s):=\sum_{n=1}^\infty \lambda_g(n) e(nx) n^{-s},\qquad\Re s>1,\]
summarized by the following lemma.

\begin{lemma}\label{Dlemma} For $(a,c)=1$ the Dirichlet series $D\left(g,\frac{a}{c},s\right)$ can be analytically continued to a meromorphic function which is holomorphic in the whole complex plane up to simple poles at $s=1\pm ir$, where the residues equal $\zeta(1\pm 2ir)/c^{1\pm 2ir}$. It satisfies the functional equation
\begin{align}\label{Dfunct}
D\left(g,\frac{a}{c},s\right)=
&\pi^{-1}2^{2s-1}\left(\frac{c}{\pi}\right)^{1-2s}\Gamma(1-s+ir)\Gamma(1-s-ir)\\
\notag&\times
\left\{\cos(\pi ir)D\left(g,\frac{\ov a}{c},1-s\right)-\cos(\pi s)D\left(g,-\frac{\ov a}{c},1-s\right)\right\}.
\end{align}
\end{lemma}

\begin{remark}
The limiting case $r\to 0$ of this lemma is due to Estermann~\cite{Es}
and served for Jutila as the starting point for his Voronoi formula, see \cite[Lemma~1]{Ju}.
Correspondingly, the limit of \eqref{vorpolar} under $r\to 0$ equals \eqref{vorpolar2}.
\end{remark}

\begin{proof}[Proof of Lemma~\ref{Dlemma}] The result follows from \cite[Lemma~3.7]{Mot} upon noting that
the Dirichlet series $D\left(g,\frac{a}{c},s\right)$ is a shift of a so-called
Estermann zeta-function:
\[D\left(g,\frac{a}{c},s-ir\right)=\sum_{n=1}^\infty\sigma_{2ir}(n)e\left(\frac{an}{c}\right)n^{-s}.\]
The proof of \cite[Lemma~3.7]{Mot} is based on analytic properties of the Hurwitz zeta-function. In the Appendix
we provide an alternate proof of Lemma~\ref{Dlemma}, based on the modularity of $g=E_r$, which displays the similarity with the cuspidal case.
\end{proof}

With Lemma~\ref{Dlemma} at hand it is straightforward to deduce \eqref{vor} with the additional polar term \eqref{vorpolar} on the right hand side. The left hand side of \eqref{vor} equals
\[c\sum_{n=1}^\infty \lambda_g(n)e\biggl(n\frac{a}{c}\biggr) F(n) =
\frac{c}{2\pi i}\int_{(2)}\widecheck F(s)\,D\left(g,\frac{a}{c},s\right)ds,\]
where $\widecheck F$ denotes the Mellin transform of $F$. We shift the contour to the vertical line at $-1$ and record the contribution of the residues at $s=1\pm ir$; we obtain
\[c\sum_{n=1}^\infty \lambda_g(n)e\biggl(n\frac{a}{c}\biggr) F(n) = \sum_{\pm}\frac{\zeta(1\pm 2ir)}{c^{\pm 2ir}}\widecheck F(1\pm ir)+
\frac{c}{2\pi i}\int_{(-1)}\widecheck F(s)\,D\left(g,\frac{a}{c},s\right)ds.\]
We observe that the sum on the right hand side equals \eqref{vorpolar}, hence we are left with proving that the remaining term equals the right hand side of \eqref{vor}. By \eqref{Dfunct} the term in question can be rewritten as
\begin{align*}\frac{1}{\pi i}&\int_{(-1)}\widecheck F(s)\left(\frac{c}{2\pi}\right)^{2-2s}
\Gamma(1-s+ir)\Gamma(1-s-ir)\\&\times
\left\{\cos(\pi ir)D\left(g,\frac{\ov a}{c},1-s\right)-\cos(\pi s)D\left(g,-\frac{\ov a}{c},1-s\right)\right\}ds.\end{align*}
We apply the change of variable $s\to 1-\frac{s}{2}$ and unfold the series $D\left(g,\pm\frac{a}{c},s\right)$.
By absolute convergence we see that the previous display equals
\begin{align*}
&\sum_{n=1}^\infty \lambda_{g}(n) e\biggl(+n \frac{\ov{a}}{c}\biggr)\frac{1}{2\pi i}
\int_{(2)}\left(\frac{2\pi\sqrt{n}}{c}\right)^{-s}\ \cos(\pi ir)\ \Gamma\left(\frac{s+2ir}{2}\right)\Gamma\left(\frac{s-2ir}{2}\right)\widecheck F\left(1-\frac{s}{2}\right)ds\\
-&\sum_{n=1}^\infty \lambda_{g}(n) e\biggl(-n \frac{\ov{a}}{c}\biggr)\frac{1}{2\pi i}
\int_{(2)}\left(\frac{2\pi\sqrt{n}}{c}\right)^{-s}\cos\left(\frac{\pi s}{2}\right)\Gamma\left(\frac{s+2ir}{2}\right)\Gamma\left(\frac{s-2ir}{2}\right)\widecheck F\left(1-\frac{s}{2}\right)ds.
\end{align*}
By \cite[6.8.17~\&~6.8.26]{Er} and \eqref{Jdef} this is the same as
\[\sum_{n=1}^\infty \lambda_{g}(n) \sum_{\pm} e\biggl(\mp n \frac{\ov{a}}{c}\biggr)
\frac{1}{2\pi i}
\int_{(2)}\left(\frac{4\pi\sqrt{n}}{c}\right)^{-s}\widecheck{J_g^{\pm}}(s)\,\widecheck F\left(1-\frac{s}{2}\right)ds,\]
which is precisely the right hand side of \eqref{vor}, thanks to the following identity for $t>0$:
\begin{align*}
\frac{1}{2\pi i}\int_{(2)}t^{-s}\widecheck{J_g^{\pm}}(s)\,\widecheck F\left(1-\frac{s}{2}\right)ds
&=\frac{1}{2\pi i}\int_{(2)}t^{-s}\widecheck{J_g^{\pm}}(s)\left\{\int_0^\infty F(x)x^{-s/2}\,dx\right\}ds\\
&=\int_0^\infty F(x)\left\{\frac{1}{2\pi i}\int_{(2)}\left(t\sqrt{x}\right)^{-s}\widecheck{ J_g^{\pm}}(s)\,ds\right\}dx\\
&=\int_0^\infty F(x)\,J_g^\pm\left(t\sqrt{x}\right)dx.
\end{align*}
The proof of Proposition~\ref{voro} is complete.
\end{proof}

\subsection{Bessel functions}

We start with some standard bounds; more precise results can be found in \cite[Appendix]{HM}.
It follows from the power series expansion that
\begin{equation}\label{Besselsmall}
  J_{k-1}(x) \ll \frac{x^{k-1}}{\Gamma(k)}, \qquad 0<x \leq 1,
\end{equation}
uniformly in $k \in \NN$. It follows from the asymptotic formula and the power series expansion that
\begin{equation}\label{Bessellarge}
  J_{k-1}(x),\ Y_{2 ir}(x) \ll x^{-1/2}, \qquad x > 0,
\end{equation}
for fixed $k\in\NN$ and $r \in \RR$. Finally we recall the bound
\begin{equation}\label{besselK}
  K_{2 ir}(x)  \ll e^{-x}, \qquad x \geq 1,
\end{equation}
for fixed $r \in \RR$, which follows again from the asymptotic formula.

For a smooth function $h:(0,\infty)\to\CC$ of compact support let
\begin{equation*}
  \widehat{h}(t) := \int_{0}^{\infty} h(x) e(xt) \,dx, \qquad \tilde{h}(t) := \frac{1}{\sqrt{2\pi}} \int_0^{\infty} h(\sqrt{x}) e\left(\frac{xt}{2\pi}\right) \frac{dx}{\sqrt{x}}.
\end{equation*}
Then the following summation formula holds.

\begin{lemma}\label{lem2} Let $K>0$, $\xi>0$, and $h:(0,\infty)\to\RR$ a smooth function of compact support. Then
\begin{displaymath}
\begin{split}
\sum_{k \equiv 0 \, (2)} i^{-k}h\left(\frac{k-1}{K}\right) J_{k-1}(\xi) =
-\frac{K}{2\sqrt{\xi}} \Im\left(e\left(\frac{\xi}{2\pi} -\frac{1}{8}\right)\tilde{h}\left(\frac{K^2}{2\xi}\right)\right) + O\left( \frac{\xi}{K^4} \int_{-\infty}^{\infty} |\widehat{h}(t) t^4| \,dt \right)
\end{split}
\end{displaymath}
with an absolute implied constant.
\end{lemma}

\begin{remark} For odd weights one has the similar formula
\begin{displaymath}
\begin{split}
\sum_{k \equiv 1 \, (2)} i^{-k}h\left(\frac{k-1}{K}\right) J_{k-1}(\xi) = -\frac{i K}{2\sqrt{\xi}} \Re\left(e\left(\frac{\xi}{2\pi} -\frac{1}{8}\right)\tilde{h}\left(\frac{K^2}{2\xi}\right)\right) + O\left(\frac{\xi}{K^4} \int_{-\infty}^{\infty} |\widehat{h}(t) t^4| \,dt \right).
\end{split}
\end{displaymath}
\end{remark}

\begin{remark} It is well-known that an individual Bessel function $J_{k-1}(x)$ is hard to
control in the range $k \ll x \ll k^2$. The lemma shows, however, that on average over
$k$ its values cancel almost completely until $x \approx k^2$, when the asymptotic
behavior becomes stable. For further discussion of the error term see
\cite[p.\ 87]{Iw}.
\end{remark}

\begin{proof} This is based on \cite[Lemma~5.8]{Iw} and can be found in  \cite[Lemma~2.3]{Kh}.
\end{proof}

Finally we state a result which is useful for estimating the Hankel-type transform occurring in the Voronoi formula \eqref{vor}.

\begin{lemma}\label{lemma4} Let $F:(0,\infty)\to\CC$ be a smooth function of compact support. For $s\in\CC$ let $B_s$ denote either of the Bessel functions $J_s$, $Y_s$ or $K_s$. Then for $\alpha>0$ and $j\in\NN$ we have
\[\int_0^\infty F(x)B_s(\alpha\sqrt{x})\,dx = \pm\left(\frac{2}{\alpha}\right)^j
\int_0^\infty\frac{\partial^j}{\partial x^j}\bigl(F(x)x^{-\frac{s}{2}}\bigr)x^{\frac{s+j}{2}}B_{s+j}(\alpha\sqrt{x})\,dx.\]
\end{lemma}

\begin{proof} The Bessel functions $B_s$ satisfy the recurrence relation $\bigl(x^s B_s(x)\bigr)'=\pm x^s B_{s-1}$
which translates to
\[(\alpha\sqrt{x})^sB_s(\alpha\sqrt{x})=\pm\frac{2}{\alpha^2}\frac{\partial}{\partial x}\bigl((\alpha\sqrt{x})^{s+1}B_{s+1}(\alpha\sqrt{x})\bigr).\]
Using this identity and applying integration by parts $j$ times we obtain
\begin{align*}
\int_0^\infty F(x)B_s(\alpha\sqrt{x})\,dx
&=\pm\left(\frac{2}{\alpha^2}\right)^j\int_0^\infty F(x)(\alpha\sqrt{x})^{-s}\frac{\partial^j}{\partial x^j}\bigl((\alpha\sqrt{x})^{s+j}B_{s+j}(\alpha\sqrt{x})\bigr)\,dx\\
&=\pm\left(\frac{2}{\alpha^2}\right)^j\int_0^\infty \frac{\partial ^j}{\partial x^j}\bigl(F(x)(\alpha\sqrt{x})^{-s}\bigr)(\alpha\sqrt{x})^{s+j}B_{s+j}(\alpha\sqrt{x})\,dx\\
&=\pm\left(\frac{2}{\alpha}\right)^j
\int_0^\infty\frac{\partial^j}{\partial x^j}\bigl(F(x)x^{-\frac{s}{2}}\bigr)x^{\frac{s+j}{2}}B_{s+j}(\alpha\sqrt{x})\,dx.
\end{align*}
\end{proof}

\subsection{Fourier coefficients of cusp forms}

For a Dirichlet character $\chi$ modulo $N$ and a positive integer $c$ divisible by $N$ let
\begin{displaymath}
  S_{\chi}(m, n, c) := \asterisk\sum_{d \,(\mod c)}\chi(d)\, e\left(\frac{m\bar{d} + nd}{c}\right)
\end{displaymath}
be the twisted Kloosterman sum. We formulate Petersson's formula for the basis of holomorphic cusp forms specified in \eqref{holobasis}.

\begin{lemma}\label{lemma2} For $k, m, n \in \NN$ and $k \geq 3$ we have
\begin{displaymath}
  \sum_{j=1}^{\theta_{k}(N, \chi)}\rho(f_{j,k})\lambda_{j, k}(m)\ov{\lambda_{j, k}(n)} = \delta_{m, n} + 2\pi i^{-k}\sum_{N \mid c} \frac{S_{\chi}(m, n, c)}{c} J_{k-1}\left(\frac{4\pi\sqrt{mn}}{c}\right).
\end{displaymath}
 \end{lemma}

\begin{proof}
This follows from \cite[Proposition~14.5]{IK}.
\end{proof}

We shall often use the following standard result.

\begin{lemma}
We have the uniform bound
 \begin{equation}\label{boundfourier}
   \sum_{m \leq M}|\lambda_g(m) |^2 \ll_{g, \eps}   M^{1+\eps}.
\end{equation}
\end{lemma}

\begin{proof} See \cite[(1.80)~\&~(14.56)]{IK} for stronger bounds.
\end{proof}

\section{The diagonal term}

We shall assume \eqref{assumption} for the rest of the paper, since otherwise the asymptotic formula is trivial.
We substitute the approximate functional equation \eqref{approxfuncteq2} in the special case \eqref{para1}--\eqref{para3} into  \eqref{defI} and choose $M:=3$.
We note that both $G = G_{t, k}$ and $C = C_{t, k}$ depend (mildly) on $t$ and $k$, but in the support of $W_1$ and $W_2$ we have $C_{t,k}\asymp\tilde C:=(T+K)^8$. By \eqref{dim}--\eqref{defrho} (or directly by Lemma~\ref{lemma2}) the error term contributes at most
\begin{equation}\label{error1}
 \ll  TK \tilde C^{1/4+\eps} (T+K)^{-M} \ll (TK)^{1+\eps}(T+K)^{-1}.
\end{equation}
After inserting the main term, the $j$-sum in \eqref{defI} equals
\begin{equation}\label{jsum}
2 \sum_{n  } G\left(\frac{n}{\sqrt{C}}\right)\frac{1}{\sqrt{n}} \sum_{d_1^2d_2^2m_1m_2 = n} \frac{\chi(d_1)}{d_1^{2it}} \frac{\ov{\chi(d_2)}}{d_2^{-2it}} \frac{\lambda_g(m_1)}{m_1^{it}} \frac{\ov{\lambda_g(m_2)}}{m_2^{-it}} \sum_{j= 1}^{\theta_k(N, \chi)} \rho(f_{j,k}) \lambda_{j, k}(m_1) \ov{\lambda_{j, k}(m_2)}.
\end{equation}
By Lemma~\ref{lemma2}, the innermost sum equals
\begin{equation}\label{applicationPetersson}
  \delta_{m_1, m_2} + 2 \pi i^{-k} \sum_{N \mid c} \frac{S_{\chi}(m_1, m_2, c)}{c} J_{k-1} \left(\frac{4 \pi\sqrt{ m_1m_2}}{c}\right).
\end{equation}
The $\delta$-term contributes
\begin{displaymath}
  2\int_{0}^{\infty}W_1\left(\frac{t}{T}\right)\sum_{k \equiv 0 \, (2)} W_2\left(\frac{k-1}{K}\right)
 \sum_{d_1, d_2, m} \frac{\chi(d_1)}{d_1^{1+2it}} \frac{\ov{\chi(d_2)}}{d_2^{1-2it}} \frac{|\lambda_g(m)|^2}{m}\,G\left(\frac{d_1^2d_2^2m^2}{\sqrt{C}}\right)\,dt.
\end{displaymath}
Let $\widecheck{G}$ denote the Mellin transform of $G$ and similarly for other functions.  By \cite[p.~145]{Li} we have
\begin{displaymath}
  \sum_{m=1}^\infty  \frac{|\lambda_g(m)|^2}{m^s} = \frac{L(s, g \otimes \tilde g) } {\zeta (2s)}
  \end{displaymath}
when $g$ is cuspidal. If $g = E_{r}$ is an Eisenstein series, we take the preceding display as the definition for $L(s, g\otimes \tilde g)$, so that $L(s, E_{r} \otimes \widetilde{E_{r}}) = \zeta(s)^2 \zeta(s + 2ir)\zeta(s - 2ir)$. With this notation the contribution of the $\delta$-term in Petersson's formula equals
\begin{equation}\label{afterPet}
\begin{split}
 &2\int_{0}^{\infty}W_1\left(\frac{t}{T}\right)\sum_{k \equiv 0 \, (2)} W_2\left(\frac{k-1}{K}\right)\\
 &\times \frac{1}{2\pi i}\int_{(1)} L(1+2it+2s, \chi ) L(1-2it+2s, \ov{\chi})
 \frac{L(1+2s, g \otimes \tilde g) }{\zeta(2+4s)} \,C^{s/2} \widecheck{G}(s) \,ds \, dt.
    \end{split}
\end{equation}
We shift the contour to the line $\Re s = -1/4 + \eps$, then the new integral contributes (cf.\ \eqref{boundsG})
\begin{equation}\label{error2}
  \ll (TK)^{1+\eps} (T+K)^{-1}.
\end{equation}

There is pole at $s=0$ whose order $v$ depends on $g$: if $g$ is a cusp form then $v=2$, if $g = E_{r}$ with $r \neq 0$ then $v=3$, while for $g=E_0$ we have $v=5$. The residue of the pole is given by a linear combination (with coefficients depending at most on $N$ and $g$) of
\begin{equation}\label{error2a}
\begin{split}
  \int_{0}^{\infty} W_1\left(\frac{t}{T}\right)&  \sum_{k \equiv 0 \, (2)} W_2\left(\frac{k-1}{K}\right) L^{(j_1)} (1+ 2it, \chi) L^{(j_2)}(1-2it, \ov{\chi}) \log^{j_3}   \mathcal{C}(t, k) \,dt \\
  & + O\left((TK)^{1+\eps}(T+K)^{-1}\right)
  \end{split}
\end{equation}
for $j_1 + j_2 + j_3 \leq v-1$.  The error term comes from approximating the square root of  $C=C_{t,k}$ in \eqref{para3} by \eqref{simplecond} and inserting \eqref{laurent}. The constant in front of the leading term $(j_1, j_2, j_3) = (0, 0, v-1)$ equals (cf.\ \eqref{laurent})
\begin{equation}\label{coefficient}
\frac{2}{(v-1)!\zeta(2)}\lim_{s\to 0}\bigl(s^{v-1}L(1+2s, g \otimes \tilde g)\bigr)=
  \begin{cases}
    \frac{L(1, \text{Ad}^2g)}{\zeta(2)}, & g \text{ cuspidal,}\\
    \frac{|\zeta(1+2ir)|^2}{4\zeta(2)}, & g = E_r, \quad r \neq 0,\\
    \frac{1}{192\zeta(2)}, & g = E_0.
  \end{cases}
\end{equation}
For $\chi$ trivial (i.e.\ $N$=1) there are also poles at $s=\pm it$, but their contribution is (cf.\ \eqref{boundsG})
\begin{equation}\label{error5}
  \ll (TK)^{1+\eps}T^{-A}.
\end{equation}
If $g=E_r$ with $r \neq  0$, there are two additional simple poles in \eqref{afterPet} at $\pm ir$ with residue
\begin{equation}\label{extrares}
\begin{split}
 & 2\int_0^{\infty} W_1\left(\frac{t}{T}\right) \sum_{k\equiv 0 \, (2)} W_2\left(\frac{k-1}{K}\right)  \\
 & \times \frac{L(1 + 2it \pm 2 ir, \chi) L(1 - 2 i t \pm 2 i r, \bar{\chi})\zeta(1 \pm 2 ir)^2 \zeta(1\pm 4 ir)}{2 \zeta(2 \pm 4 ir)} C^{\pm ir/2} \widecheck{G}(\pm ir) \,dt.
  \end{split}
\end{equation}
Here again we can approximate $C^{1/2}$ by \eqref{simplecond} and $\widecheck{G}(\pm ir)$ by $\widecheck{G}_0(\pm ir)$ at the cost of an error term
$O\left((TK)^{1+\eps}(T+K)^{-1}\right)$, cf.\ \eqref{para3} and \eqref{MellinGdef}.

We continue with the analysis of \eqref{error2a}. Applying Poisson summation it is straightforward to see that the main term in \eqref{error2a} equals,
for any $A> 0$,
\begin{equation}\label{firststep}
\begin{split}
  \frac{1}{2} \int_{0}^{\infty} &\int_{0}^{\infty} W_1\left(\frac{t}{T}\right) W_2\left(\frac{x}{K}\right)  L^{(j_1)} (1+ 2it, \chi) L^{(j_2)}(1-2it, \ov{\chi}) \log^{j_3}\mathcal{C}(t, x) \,dt\, dx  \\
 & + O\left((TK )^{1+\eps} K^{-A}\right).
    \end{split}
\end{equation}
Now we use the simple approximation
\begin{displaymath}
  L^{(j)}(1+ 2it, \chi) =  \sum_{m\leq  M} \frac{\chi(m)(-\log m)^j}{m^{1+2it}} + O_{j}\left(\frac{(\log M)^j}{M}\right), \qquad M \geq \frac{e^{3j/2}}{3}N|t|,
\end{displaymath}
which follows in a standard fashion from van der Corput's lemma (see Lemma~4.10 and Theorem~4.11 in \cite{Ti}). We substitute this into \eqref{firststep} with $M := e^6NT$ (note that $j_1,j_2\leq 4$ and $|t|\leq 2T$). Then the $t$-integral equals, uniformly for $x\asymp K$,
\begin{equation*}
\sum_{m, n \leq e^6 NT} \frac{\chi(m) \ov{\chi(n)} (-\log m)^{j_1} (-\log n)^{j_2}}{mn}
   \int_{0}^{\infty} W_1\left(\frac{t}{T}\right)  \log^{j_3}\mathcal{C}(t, x) \left(\frac{n}{m}\right)^{2it} dt + O\bigl((TK)^{\eps}\bigr).
\end{equation*}
The diagonal and error terms together contribute (cf.\ \eqref{smoothlog})
\begin{equation}\label{diagonalterm}
 \frac{\zeta_{(N)}^{(j_1+j_2)}(2)}{2}TK\mathcal{L}_{j_3}(T,K)+O\bigl(K(TK)^{\eps}\bigr),
\end{equation}
where $(N)$ indicates the removal of the Euler factors at primes dividing $N$. Let us now estimate the off-diagonal term. Partial integration shows, for any $m\neq n$ and $x\asymp K$,
\begin{displaymath}
   \int_{0}^{\infty} W_1\left(\frac{t}{T}\right)  \log^{j_3}\mathcal{C}(t, x)\left(\frac{n}{m}\right)^{2it} dt \ll \frac{(TK)^{\eps}}{|\log(n/m)|},
\end{displaymath}
hence the off-diagonal term contributes at most
\begin{align}
\nonumber&\ll K(TK)^\eps\sum_{\substack{m,n\leq e^6 NT\\m\neq n}}\frac{1}{mn|\log(n/m)|}\\
\nonumber&\ll K(TK)^\eps\sum_{m<n\leq e^6 NT}\frac{1}{n\min(m,n-m)}\\
\label{error4}&\ll K(TK)^\eps\sum_{m',n'\leq e^6 NT}\frac{1}{m'n'}\ll K(TK)^\eps.
\end{align}

The same argument shows that \eqref{extrares} equals, up to an error term already present in \eqref{error2a},
\begin{equation}\label{extrares1}
   TK  \mathcal{M}_{\pm ir}(T,K) \frac{\zeta_{(N)}(2 \pm 4 ir)\zeta(1 \pm 2 ir)^2 \zeta(1\pm 4 ir)}{2 \zeta(2 \pm 4 ir)}\widecheck{G}_0(\pm ir).
\end{equation}

\begin{remark} We will see in Section~\ref{secpolar} that the term \eqref{extrares1} will be cancelled by some portion of the off-diagonal term. We note that we could have avoided the computation of \eqref{extrares} and \eqref{extrares1} in the case $g = E_r$ with $r \neq  0$ by choosing  $G_0$ as in \eqref{G0def}--\eqref{Hfeature} with the additional feature that $H(s)$ is divisible by $s^2+r^2$.  Then
$\widecheck G_0(\pm ir)=0$, hence by \eqref{MellinGdef} also $\widecheck G(\pm ir)=0$,
and the residues vanish. This trick was probably used for the first time in \cite{BHM}.
\end{remark}

The various error terms in \eqref{error1}, \eqref{error2}--\eqref{error2a}, \eqref{error5}, \eqref{firststep}--\eqref{error4} are admissible in Theorems~\ref{theorem1}--\ref{theorem3}. The main term in \eqref{diagonalterm} furnishes the main term of Theorem~\ref{theorem1} and part of the main terms in Theorems~\ref{theorem2}--\ref{theorem3}. We note that Section~\ref{secpolar} discusses
an additional contribution to the main terms of Theorems~\ref{theorem2}--\ref{theorem3}.
The leading constant in the various cases of $g$ follows from \eqref{coefficient} and  \eqref{diagonalterm}.

\section{The off-diagonal term}

\subsection{Averaging over $k$ and $t$} We return to \eqref{jsum} which is the $j$-sum in \eqref{defI}, and substitute now the Kloosterman term in \eqref{applicationPetersson} for the  innermost sum. This gives a total contribution of
\begin{equation}\label{totaloff}
\begin{split}
&4 \pi \int_{0}^{\infty} W_1\left(\frac{t}{T}\right)\sum_{k \equiv 0 \, (2)} W_2\left(\frac{k-1}{K}\right)  i^{-k}   \sum_{n  }
G\left(\frac{n}{\sqrt{C}}\right)\frac{1}{\sqrt{n}} \\
&\times\sum_{d_1^2d_2^2m_1m_2 = n} \frac{\chi(d_1)}{d_1^{2it}} \frac{\ov{\chi(d_2)}}{d_2^{-2it}} \frac{\lambda_g(m_1)}{m_1^{it}} \frac{\ov{\lambda_g(m_2)}}{m_2^{-it}}
\sum_{N \mid c} \frac{S_{\chi}(m_1, m_2, c)}{c} J_{k-1} \left(\frac{4 \pi\sqrt{ m_1m_2}}{c}\right) dt.
\end{split}
\end{equation}
 Recall that both $G = G_{t, k}$ and $C = C_{t, k}$,  defined in \eqref{Gdef} and \eqref{para3}, depend (mildly) on $t$ and $k$, but we have $C_{t,k}\asymp\tilde C:=(T+K)^8$. Moreover, $G$ is real-valued as we have assumed
according to Remark~\ref{selfdual}. First we observe that \eqref{totaloff} is absolutely convergent, due to the rapid decay of the Bessel $J$-function near 0 for large $k$, see \eqref{Besselsmall}.  More precisely, we can truncate the multiple sum at
\begin{displaymath}
  n \leq \tilde C^{1/2+\eps}, \qquad c \leq \sqrt{m_1m_2}\tilde C^{\eps}
\end{displaymath}
at the cost of a negligible error.

Next we write
\begin{displaymath}
  W_1\left(\frac{t}{T}\right)  W_2\left(\frac{k-1}{K}\right) G_{t, k}\left(\frac{n}{\sqrt{C_{t, k}}}\right) =: \Omega_1\left(\frac{t}{T}, \frac{k-1}{K}, \frac{n}{\sqrt{\tilde{C}}}\right)
\end{displaymath}
where $\Omega_1 : (0, \infty)^3 \rightarrow \RR$ is smooth function depending on $K, T, N$ and $g$ with the following properties: it is compactly supported in first two variables and rapidly decaying in the third variable, uniformly in $K$ and $T$. Let
\begin{equation}\label{defom1}
  \tilde{\Omega}_1(x, y, z) := W_1(x)W_2(y) G_0\left(\frac{(2\pi)^4z\sqrt{\tilde{C}}}{N^2\left((xT)^2+(\frac{1}{2}yK)^2\right)^2}\right).
\end{equation}
Then by \eqref{fateta}, \eqref{Gdef}, \eqref{para3} the partial derivatives satisfy, uniformly in
$\bm x\in(0,\infty)^3$,
\[{\bm x}^{\bm k}\left(\Omega_1^{(\bm j)}({\bm x})-\tilde\Omega_1^{(\bm j)}({\bm x})\right)\ll_{{\bm j},{\bm k}} (T+K)^{-1},\qquad {\bm j},{\bm k}\in\NN_0^3.\]

In order to sum over $k$ with the help of Lemma~\ref{lem2}, we define
\begin{displaymath}
  \Omega_2(x, u, z) := \int_0^{\infty} \Omega_1(x, y, z) e^{i y^2 u} \,dy.
\end{displaymath}
Then $\Omega_2: (0,\infty)\times\RR\times(0,\infty)\to\CC$ is compactly supported in the first variable.   Integration by parts shows that it is rapidly decaying in the other two variables, and moreover for
\begin{equation}\label{defom2}
  \tilde{\Omega}_2(x, u, z) := \int_0^{\infty} \tilde{\Omega}_1(x, y, z) e^{i y^2 u} \,dy
\end{equation}
we have, uniformly in $\bm x\in (0,\infty)\times\RR\times(0,\infty)$,
\[{\bm x}^{\bm k}\left(\Omega_2^{(\bm j)}({\bm x})-\tilde\Omega_2^{(\bm j)}({\bm x})\right)\ll_{{\bm j},{\bm k}} (T+K)^{-1},\qquad {\bm j},{\bm k}\in\NN_0^3.\]
Applying now Lemma~\ref{lem2} with $\xi := 4\pi \sqrt{m_1m_2}/c$ to \eqref{totaloff}, we obtain the main term   \begin{displaymath}
 \begin{split}
& K   \sum_{\pm} \frac{1\pm i}{2} \int_{0}^{\infty}   \sum_{d_1, d_2, m_1, m_2}   \sum_{N \mid c}   \Omega_2\left(\frac{t}{T}, \frac{\pm K^2c}{8\pi \sqrt{m_1m_2}}, \frac{d_1^2d_2^2m_1m_2}{\sqrt{\tilde{C}}}\right)\\
& \quad \times \frac{\chi(d_1)}{d_1^{1+2it}} \frac{\ov{\chi(d_2)}}{d_2^{1-2it}} \frac{\lambda_g(m_1)}{m_1^{3/4+it}} \frac{\ov{\lambda_g(m_2)}}{m_2^{3/4-it}} e\left(\pm \frac{2 \sqrt{m_1m_2}}{c}\right)
\frac{S_{\chi}(m_1, m_2, c)}{c^{1/2}}   \, dt,
\end{split}
\end{displaymath}
while the  error term in Lemma~\ref{lem2} infers a total error of
\begin{equation}\label{offerror1}
  \begin{split}
&\ll  \tilde C^{\eps}\frac{T}{K^4} \sum_{m_1m_2 \leq \tilde C^{1/2+\eps}}   |\lambda_g(m_1) \lambda_g(m_2) | \sum_{\substack{ c \leq \tilde C^{1/4+\eps}\\ N \mid c}} \frac{|S_{\chi}(m_1, m_2, c)|}{c^2}\\
&\ll  \tilde C^{\eps}\frac{T}{K^4}\sum_{mn \leq \tilde C^{1/2+\eps}}|\lambda_g(m)|^2\ll
(TK)^{\eps}\frac{T(T^4+K^4)}{K^4}\ll (TK)^{1+\eps} \left(\frac{T^4}{K^5}+\frac{K^3}{T^4}\right).
\end{split}
\end{equation}
 Here it is enough to use the trivial bound $|S_{\chi}(m_1, m_2, c)| \leq c$ and \eqref{boundfourier}.

In order to integrate over $t$, we define
\begin{displaymath}
  \Omega_3(w, u, z) := \int_0^{\infty} \Omega_2(x, u, z) e^{i wx} \,dx.
\end{displaymath}
Then $\Omega_3: \RR^2\times (0, \infty)\to\CC$ is rapidly decaying in all three variables.  Putting
\begin{equation}\label{defom3}
  \tilde{\Omega}_3(w, u, z) := \int_0^{\infty} \tilde{\Omega}_2(x, u, z) e^{i wx} \,dx,
 \end{equation}
we see as above, that uniformly in $\bm x\in\RR^2\times (0, \infty)$,
\begin{equation}\label{tilde}
{\bm x}^{\bm k}\left(\Omega_3^{(\bm j)}({\bm x})-\tilde\Omega_3^{(\bm j)}({\bm x})\right)\ll_{{\bm j},{\bm k}} (T+K)^{-1},\qquad {\bm j},{\bm k}\in\NN_0^3.
\end{equation}
In addition, we factor out $\delta=(d_1, d_2)$ to arrive at
\begin{displaymath}
\begin{split}
  TK  &  \sum_{\pm} \frac{1\pm i}{2} \sum_{\substack{(\delta, N) = 1\\(d_1, d_2) = 1}} \sum_{m_1, m_2}   \sum_{N \mid c}   \Omega_3\left(T \log \frac{d_2^2m_2}{d_1^2m_1}, \frac{\pm K^2c}{8\pi \sqrt{m_1m_2}}, \frac{\delta^4d_1^2d_2^2m_1m_2}{\sqrt{\tilde{C}}}\right)\\
& \times \frac{\chi(d_1)\ov{\chi(d_2)}}{\delta^2 d_1d_2} \frac{\lambda_g(m_1)}{m_1^{3/4}} \frac{\ov{\lambda_g(m_2)}}{m_2^{3/4}} e\left(\pm \frac{2 \sqrt{m_1m_2}}{c}\right)
\frac{S_{\chi}(m_1, m_2, c)}{c^{1/2}}.
\end{split}
\end{displaymath}
It is convenient to introduce dyadic decompositions. Let $\omega:(0,\infty)\to[0,\infty)$ be a smooth function supported on $[1/2,2]$ such that
\begin{equation}\label{dyadic}
  \sum_{j=0}^{\infty} \omega(x/2^j) = 1, \qquad x\geq 1 .
\end{equation}
Then we can recast the preceding expression as
\begin{equation}\label{transformed}
\begin{split}
TK &  \sum_{\pm} \frac{1\pm i}{2} \sum_{\substack{M_1, M_2, G\geq 1\\ \text{powers of two}}} \sum_{\substack{(\delta, N) = 1\\(d_1, d_2) = 1}} \sum_{m_1, m_2}   \sum_{N \mid c}   \Omega_3\left(T \log \frac{d_2^2m_2}{d_1^2m_1}, \frac{\pm K^2c}{8\pi \sqrt{m_1m_2}}, \frac{\delta^4d_1^2d_2^2m_1m_2}{\sqrt{\tilde{C}}}\right)\\
& \times \omega\left(\frac{m_1}{M_1}\right) \omega\left(\frac{m_2}{M_2}\right)\omega\left(\frac{c}{G}\right) \frac{\chi(d_1)\ov{\chi(d_2)}}{\delta^2 d_1d_2} \frac{\lambda_g(m_1)}{m_1^{3/4}} \frac{\ov{\lambda_g(m_2)}}{m_2^{3/4}} e\left(\pm \frac{2 \sqrt{m_1m_2}}{c}\right)
\frac{S_{\chi}(m_1, m_2, c)}{c^{1/2}}.
\end{split}
\end{equation}
By the rapid decay of $\Omega_3$ in all three variables we can restrict the summation to
\begin{equation}\label{sizes}
  M_1M_2 \leq  \frac{\tilde C^{1/2+\eps}}{\delta^4d_1^2d_2^2}, \qquad G \leq \frac{\tilde C^\eps(M_1M_2)^{1/2}}{K^2}.
\end{equation}
Moreover, we can assume
\begin{equation}\label{close}
  d_2^2m_2 = d_1^2m_1\left(1+O(T^{\eps-1})\right).
\end{equation}
We will keep this in mind for later and in particular often use
\begin{equation}\label{sizeM}
  d_2^2M_2 \asymp d_1^2M_1.
\end{equation}

\subsection{Interlude on character sums}

For later purposes we need to transform the expression
\begin{displaymath}
  S_{\chi}(m_1, m_2, c) e \left(\pm \frac{2\sqrt{m_1m_2}}{c} \right)
\end{displaymath}
in \eqref{transformed}. We proceed similarly as in \cite[Section~6]{Y}. We start by writing
\begin{displaymath}
  e \left(\pm \frac{2\sqrt{m_1m_2}}{c} \right) = e\left(\pm \frac{d_1^2m_1+d_2^2m_2}{d_1d_2c}\right) e\left(\mp \frac{(\sqrt{d_1^2m_1} - \sqrt{d_2^2m_2})^2}{d_1d_2c}\right).
\end{displaymath}
We infer
\begin{align}
 S_{\chi}(m_1, m_2, c)&e\left(\pm \frac{d_1^2m_1+d_2^2m_2}{d_1d_2c}\right) \nonumber\\
 = & \frac{\phi(c)}{\phi(d_1d_2c)} S_{\chi}(d_1d_2m_1, d_1d_2m_2, d_1d_2c)  e\left(\pm \frac{d_1^2m_1+d_2^2m_2}{d_1d_2c}\right)\nonumber\\
 = & \frac{\phi(c)}{\phi(d_1d_2c)}  \asterisk\sum_{d\, (d_1d_2c)} \chi(d) e\left(m_1\frac{d_2\bar{d} \pm d_1}{d_2c} + m_2 \frac{d_1d\pm d_2}{d_1c}\right)\label{char}.
\end{align}
We write
\begin{displaymath}
  (d_1d \pm d_2, c) = g, \qquad c = gh, \qquad d_1d \pm d_2 = gf.
\end{displaymath}
Note that $(gf,d_1d_2)=1$ by $(d_1,d_2)=1$. Let us fix a decomposition $c=gh$ with $(g,d_1d_2)=1$. It is straightforward to check, using again $(d_1,d_2)=1$, that as $d$ runs through $(\ZZ/d_1d_2c\ZZ)^{\times}$ with $(d_1d\pm d_2, c) = g$, then $f$ runs through $(\ZZ/d_1^2d_2h\ZZ)^{\times}$ with $(gf \mp d_2, d_1^2d_2h) = d_1$, and this is a bijection between the relevant residue classes. With this notation we have
\begin{displaymath}
  d_2\bar{d} \pm d_1 \equiv \pm gf \ov{(gf\mp d_2)/d_1}\pmod{d_1d_2c}.
\end{displaymath}
Hence we can recast \eqref{char} as
\begin{displaymath}
   \frac{\phi(c)}{\phi(d_1d_2c)} \sum_{\substack{gh = c\\(g, d_1d_2) = 1}}  \asterisk\sum_{\substack{f\, (d_1^2d_2h)\\ (gf \mp d_2, d_1^2d_2h) = d_1}} \chi\left(\frac{gf\mp d_2}{d_1}\right) e\left(\pm m_1\frac{f\ov{(gf \mp d_2)/d_1} }{d_2h} + m_2 \frac{f}{d_1h} \right).
\end{displaymath}
Using  the change of variable $f \mapsto \bar{f}$, we can summarize the preceding   discussion as
\begin{equation}\label{finalchar}
\begin{split}
 &  S_{\chi}(m_1, m_2, c) e \left(\pm \frac{2\sqrt{m_1m_2}}{c} \right) = e\left(\mp \frac{(\sqrt{d_1^2m_1} - \sqrt{d_2^2m_2})^2}{d_1d_2c}\right) \frac{\phi(c)}{\phi(d_1d_2c)} \\
 &\times  \sum_{\substack{gh = c\\(g, d_1d_2) = 1}}  \asterisk\sum_{\substack{f\, (d_1^2d_2h)\\ (g\bar{f} \mp d_2, d_1^2d_2h) = d_1}} \chi\left( \frac{g\bar{f}\mp d_2}{d_1}\right) e\left(\pm m_1\frac{\ov{(g \mp d_2f)/d_1} }{d_2h} + m_2 \frac{\bar{f}}{d_1h} \right).
 \end{split}
\end{equation}

For later purposes, we prove the following two lemmata.
\begin{lemma}\label{lem6} Let $d_1, d_2, g, h, m_1, m_2$ be positive integers such that $gh$ is divisible by $N$ and $d_1, d_2, g$ are pairwise coprime. Then
\[\Biggl|\asterisk\sum_{\substack{f\, (d_1^2d_2h)\\ (g \mp d_2f, d_1^2d_2h) = d_1}} \chi\left(\frac{g\bar{f} \mp d_2}{d_1}\right) e\left( (m_1 - m_2)\frac{f}{d_1h} \right)\Biggr|
\leq Nd_1d_2\tau(h)^2(h,(m_1-m_2)d_1d_2).\]
\end{lemma}

\begin{proof} We use first that $\chi$ is a character modulo $N$ which divides $gh$.
Writing
\[h=h_1h_2,\qquad h_1:=\frac{N}{(N,g)},\]
we see that the character sum in question equals
\[\asterisk\sum_{\substack{f_1\, (d_1^2d_2h_1)\\ (g \mp d_2f_1, d_1^2d_2h_1) = d_1}} \chi\left(\frac{g\bar{f_1} \mp d_2}{d_1}\right)
  \asterisk\sum_{\substack{f\, (d_1^2d_2h)\\ (g \mp d_2f, d_1^2d_2h) = d_1\\ f\equiv f_1\,(d_1^2d_2h_1)}}
  e\left( (m_1 - m_2)\frac{f}{d_1h} \right).\]
With the notation
\[f=:f_1+kd_1^2d_2h_1,\qquad g_1:=\frac{g\mp d_2f_1}{d_1}\]
this becomes
\[\asterisk\sum_{\substack{f_1\, (d_1^2d_2h_1)\\ (g \mp d_2f_1, d_1^2d_2h_1) = d_1}}
\chi\left(\frac{g\bar{f_1} \mp d_2}{d_1}\right) e\left( (m_1 - m_2)\frac{f_1}{d_1h} \right)
\sum_{\substack{k\, (h_2)\\ (f_1+kd_1^2d_2h_1,h_2)=1\\ (g_1 \mp kd_1d_2^2h_1,h_2)=1}}
e\left( (m_1 - m_2)d_1d_2\frac{k}{h_2} \right).\]
The $k$-sum equals
\[\sum_{\substack{\ell_1\mid h_2\\\ell_2\mid h_2}}\mu(\ell_1)\mu(\ell_2)\sum_{\substack{k\,(h_2)\\ kd_1^2d_2h_1\equiv -f_1\,(\ell_1)\\ kd_1d_2^2h_1\equiv \pm g_1\,(\ell_2)}}e\left( (m_1 - m_2)d_1d_2\frac{k}{h_2} \right).\]
It is straightforward to evaluate the inner sum explicitly, but for our purposes it suffices to record that it is at most $(h_2,(m_1-m_2)d_1d_2)$ which gives the lemma.
\end{proof}

\begin{lemma} For $r \in \RR$ and $\Re s>1/2$ let
\begin{equation}\label{defXi}
\Xi_r^\pm(s):=\sum_{\substack{(\delta,N)=1\\(d_1,d_2)=1}}
\frac{\chi(d_1)\ov{\chi(d_2)}}{\delta^{\lambda+\mu}(d_1d_2)^\lambda}
\sum_{\substack{gh\equiv 0\,(N)\\(g,d_1d_2)=1}}
\frac{1}{g^\mu h^\lambda}
\frac{\phi(gh)}{\phi(d_1d_2gh)}\asterisk\sum_{\substack{f\,(d_1^2d_2h)\\(gf\mp d_2,d_1^2d_2h)=d_1}}\chi\left(\frac{gf\mp d_2}{d_1}\right),
\end{equation}
where $\chi$ is an even primitive Dirichlet character modulo $N$ and
\[\lambda:=2+2s+2ir,\qquad\mu:=2s-2ir.\]
Then
\begin{equation}\label{finalXi}
  \Xi_r^\pm(s) = \zeta(2s-2ir)\zeta(1+2s+2ir)N^{-2s+2ir} \prod_{p \mid N} (1-p^{-2-4ir}) .
\end{equation}
\end{lemma}

\begin{proof}
By the remarks below \eqref{char} the innermost sum in \eqref{defXi} can be rewritten as
\[\asterisk\sum_{\substack{f\,(d_1^2d_2h)\\(gf\mp d_2,d_1^2d_2h)=d_1}}\chi\left(\frac{gf\mp d_2}{d_1}\right)
=\asterisk\sum_{\substack{d\,(d_1d_2gh)\\(d_1d\pm d_2,gh)=g}}\chi(d).\]
On the right hand side the condition on $d$ only depends on $d\;\mod gh$, therefore
\[\frac{\phi(gh)}{\phi(d_1d_2gh)}\asterisk\sum_{\substack{f\,(d_1^2d_2h)\\(gf\mp d_2,d_1^2d_2h)=d_1}}\chi\left(\frac{gf\mp d_2}{d_1}\right)
=\asterisk\sum_{\substack{d\,(gh)\\(d_1d\pm d_2,gh)=g}}\chi(d).\]
With the notation
\[S(\chi,g,h,d_1,d_2):=\chi(d_1)\ov{\chi(d_2)}\asterisk\sum_{\substack{d\,(gh)\\(d_1d\pm d_2,gh)=g}}\chi(d)\]
we can now write
\[\Xi_r^\pm(s)=\sum_{(\delta,N)=1}\frac{1}{\delta^{\lambda+\mu}}
\sum_{\substack{(d_1d_2,N)=1\\(d_1,d_2)=1}}\frac{1}{(d_1d_2)^\lambda}
\sum_{\substack{gh\equiv 0\,(N)\\(g,d_1d_2)=1}}
\frac{1}{g^\mu h^\lambda}S(\chi,g,h,d_1,d_2).\]
The character sum $S(\chi,g,h,d_1,d_2)$ is multiplicative in $\chi,g,h$ in the following sense. If $g=g_1g_2$
and $h=h_1h_2$ are any decompositions such that $(g_1h_1,g_2h_2)=1$, and correspondingly $\chi=\chi_1\chi_2$ where $\chi_i$ is a Dirichlet character modulo $g_ih_i$, then
\[S(\chi,g,h,d_1,d_2)=S(\chi_1,g_1,h_1,d_1,d_2)S(\chi_2,g_2,h_2,d_1,d_2).\]
Indeed, this follows easily upon writing the summation variable $d\;\mod gh$ as
\[d=\tilde d_1e_1+\tilde d_2e_2,\qquad \tilde d_i\;\mod g_ih_i,\] where
$e_1$ and $e_2$ are fixed integers such that
\begin{align*}
&e_1\equiv 1\,(\mod g_1h_1),\qquad e_1\equiv 0\,(\mod g_2h_2);\\
&e_2\equiv 0\,(\mod g_1h_1),\qquad e_2\equiv 1\,(\mod g_2h_2).
\end{align*}
In other words, we have a decomposition over the primes
\[S(\chi,g,h,d_1,d_2)=\prod_p S(\chi_p,g_p,h_p,d_1,d_2),\]
where the subscript $p$ denotes the $p$-part.

Let us fix $p$ for a moment and use the notation
\[g_p=p^\gamma,\qquad h_p=p^\delta,\qquad N_p=p^\nu.\]
Note that $\chi_p$ is a primitive Dirichlet character modulo $p^\nu$.\\[6pt]
{\it Case 1.} If $\nu=0$ (i.e. $p\nmid N$), then $\chi_p$ is trivial, so that
\[S(\chi_p,g_p,h_p,d_1,d_2)=
\asterisk\sum_{\substack{d\,(p^{\gamma+\delta})\\(d_1d\pm d_2,p^{\gamma+\delta})=p^\gamma}}1
=\begin{cases}
1,&\delta=0,\\
p^\delta-p^{\delta-1},&\delta>0\text{ and }p\mid d_1d_2p^\gamma,\\
p^\delta-2p^{\delta-1},&\delta>0\text{ and }p\nmid d_1d_2p^\gamma.
\end{cases}\]
For $\delta=0$ the right hand side follows by observing that there is a unique $d$
satisfying the condition, since $\gamma>0$ implies $p\nmid d_1d_2$. For $\delta>0$ and $p\mid d_1d_2$
the right hand side follows by observing that $\gamma=0$ and exactly one of $d_1$ and $d_2$ is divisible by $p$, hence the condition on $d$ is automatically satisfied. For $\delta>0$ and $\gamma>0$ the right hand side follows by observing that $p\nmid d_1d_2$, so that the $d$'s satisfying the condition are in bijection with the reduced residues modulo $p^\delta$. For $\delta>0$, $\gamma=0$, and $p\nmid d_1d_2$ the right hand side follows by observing that the condition on $d$ is $(d_1d\pm d_2,p)=1$ and there are precisely $p^{\delta-1}$ reduced residue classes modulo $p^\delta$ that do not have this property.\\[6pt]
{\it Case 2.} If $\nu>0$ (i.e. $p\mid N$), then $p\nmid d_1d_2$, $\gamma+\delta\geq\nu$, and $\chi_p$ induces a nontrivial character modulo $p^{\gamma+\delta}$ of conductor $p^\nu$, so that
\begin{align*}
\chi_p(\mp 1)S(\chi_p,g_p,h_p,d_1,d_2)
&=\asterisk\sum_{\substack{d\,(p^{\gamma+\delta})\\(d_1d\pm d_2,p^{\gamma+\delta})=p^\gamma}}\chi_p(\mp \ov{d_2}d_1d)
=\asterisk\sum_{\substack{d\,(p^{\gamma+\delta})\\(d-1,p^{\gamma+\delta})=p^\gamma}}\chi_p(d)\\
&=\begin{cases}
1,&\delta=0\text{ and }\gamma\geq\nu,\\
p^\delta-p^{\delta-1},&\delta>0\text{ and }\gamma\geq\nu,\\
-p^{\delta-1},&\delta>0\text{ and }\gamma=\nu-1,\\
0,&\delta>0\text{ and }\gamma<\nu-1.
\end{cases}
\end{align*}
For $\delta=0$ and $\gamma\geq\nu$ the right hand side follows by observing that there is a unique $d$
satisfying the condition in the second sum, and this $d$ is congruent to $1$ modulo $p^\nu$. For $\delta>0$ and $\gamma\geq\nu$ the right hand side follows by observing that the $d$'s satisfying the condition in the second sum
are in bijection with the reduced residues modulo $p^\delta$, and all these $d$'s are congruent to $1$ modulo $p^\nu$. For $\delta>0$ and $\gamma=\nu-1$ the right hand side follows by observing that the condition on $d$ only depends on $d\;\mod p^\nu$, hence in this case
\[\asterisk\sum_{\substack{d\,(p^{\gamma+\delta})\\(d-1,p^{\gamma+\delta})=p^\gamma}}\chi_p(d)=
p^{\delta-1}\asterisk\sum_{\substack{d\,(p^\nu)\\(d-1,p^\nu)=p^{\nu-1}}}\chi_p(d),\]
and the sum on the right hand side equals $-1$. Indeed, for $\nu=1$ this is obvious, while for $\nu>1$ it follows from the fact that $u\mapsto\chi_p(p^{\nu-1}u+1)$ is a nontrivial additive character modulo $p$. Finally, for $\delta>0$ and $\gamma<\nu-1$ the right hand side follows by observing that the second sum does not change when we multiply it with a complex number of the form $\chi_p(p^{\nu-1}u+1)\neq 1$.
\medskip

Our findings imply that
\[S(\chi_p,g_p,h_p,d_1,d_2)=S(\chi_p,g_p,h_p,d_{1,p},d_{2,p}),\]
hence we have an Euler product decomposition
\[\Xi_r^\pm(s)=\prod_p\Xi_{r,p}^\pm(s),\]
where for $p\nmid N$
\[\Xi_{r,p}^\pm(s):=
\frac{1}{1-p^{-\lambda-\mu}}
\sum_{\substack{\alpha,\beta\geq 0\\\min(\alpha,\beta)=0}}\frac{1}{p^{(\alpha+\beta)\lambda}}
\sum_{\substack{\gamma,\delta\geq 0\\\min(\gamma,\alpha+\beta)=0}}
\frac{1}{p^{\gamma\mu+\delta\lambda}}S(\chi_p,p^\gamma,p^\delta,p^\alpha,p^\beta),\]
while for $p\mid N$
\[\Xi_{r,p}^\pm(s):=\chi_p(\mp 1)\sum_{\substack{\gamma,\delta\geq 0\\\gamma+\delta\geq\nu}}
\frac{1}{p^{\gamma\mu+\delta\lambda}}S(\chi_p,p^\gamma,p^\delta,1,1).\]
Now we insert the above calculated values of $S(\chi_p,p^\gamma,p^\delta,p^\alpha,p^\beta)$.\\[6pt]
{\it Case 1.} If $p\nmid N$, then we obtain
\begin{align*}(1-p^{-\lambda-\mu})\Xi_{r,p}^\pm(s)
&=\sum_{\gamma=0}^\infty\frac{1}{p^{\gamma\mu}}+\frac{p-2}{p}\sum_{\delta=1}^\infty\frac{1}{p^{\delta(\lambda-1)}}
+\frac{p-1}{p}\sum_{\gamma=1}^\infty\sum_{\delta=1}^\infty\frac{1}{p^{\gamma\mu+\delta(\lambda-1)}}\\
&+\left(\sum_{\alpha=1}^\infty\frac{1}{p^{\alpha\lambda}}+\sum_{\beta=1}^\infty\frac{1}{p^{\beta\lambda}}\right)
\left(1+\frac{p-1}{p}\sum_{\delta=1}^\infty\frac{1}{p^{\delta(\lambda-1)}}\right).
\end{align*}
Indeed, the first line contains the contribution of $\alpha=\beta=0$, and the second line contains the rest (where $\gamma$ must be zero). We sum all the geometric series:
\begin{align*}
(1-&p^{-\lambda-\mu})\Xi_{r,p}^\pm(s)\\
&=\frac{1}{1-p^{-\mu}}+\frac{p-2}{p^\lambda(1-p^{1-\lambda})}+\frac{p-1}{p^{\lambda+\mu}(1-p^{1-\lambda})(1-p^{-\mu})}
+\frac{2}{p^\lambda(1-p^{-\lambda})}\left(1+\frac{p-1}{p^\lambda(1-p^{1-\lambda})}\right)\\
&=\frac{(1-p^{1-\lambda})+p^{-\lambda}(p-2)(1-p^{-\mu})+p^{-\lambda-\mu}(p-1)}{(1-p^{1-\lambda})(1-p^{-\mu})}
+\frac{2p^{-\lambda}}{1-p^{1-\lambda}}\\
&=\frac{1-2p^{-\lambda}+p^{-\lambda-\mu}}{(1-p^{1-\lambda})(1-p^{-\mu})}+\frac{2p^{-\lambda}-2p^{-\lambda-\mu}}{(1-p^{1-\lambda})(1-p^{-\mu})}\\
&=\frac{1-p^{-\lambda-\mu}}{(1-p^{1-\lambda})(1-p^{-\mu})}.
\end{align*}
\medskip
Hence for $p\nmid N$ we have
\[\Xi_{r,p}^\pm(s)=\frac{1}{(1-p^{1-\lambda})(1-p^{-\mu})}.\]
\\[6pt]
{\it Case 2.} If $p\mid N$, then we obtain for $p^\nu\parallel N$
\[\chi_p(\mp 1)\Xi_{r,p}^\pm(s)=\sum_{\gamma=\nu}^\infty\frac{1}{p^{\gamma\mu}}
-\frac{1}{p^{1+(\nu-1)\mu}}\sum_{\delta=1}^\infty\frac{1}{p^{\delta(\lambda-1)}}
+\frac{p-1}{p}\sum_{\gamma=\nu}^\infty\sum_{\delta=1}^\infty\frac{1}{p^{\gamma\mu+\delta(\lambda-1)}}.\]
We sum all the geometric series:
\begin{align*}
\chi_p(\mp 1)\Xi_{r,p}^\pm(s)
&=\frac{1}{p^{\nu\mu}(1-p^{-\mu})}-\frac{1}{p^{\lambda+(\nu-1)\mu}(1-p^{1-\lambda})}
+\frac{p-1}{p^{\lambda+\nu\mu}(1-p^{1-\lambda})(1-p^{-\mu})}\\
&=\frac{(1-p^{1-\lambda})-p^{-\lambda+\mu}(1-p^{-\mu})+p^{-\lambda}(p-1)}{p^{\nu\mu}(1-p^{1-\lambda})(1-p^{-\mu})}\\
&=\frac{1-p^{-\lambda+\mu}}{p^{\nu\mu}(1-p^{1-\lambda})(1-p^{-\mu})}.
\end{align*}
Hence for $p\mid N$ we have
\[\Xi_{r,p}^\pm(s)=\chi_p(\mp 1)p^{-\nu\mu}\frac{1-p^{-\lambda+\mu}}{(1-p^{1-\lambda})(1-p^{-\mu})}.\]
Collecting the above results we arrive at \eqref{finalXi}.
\end{proof}

\subsection{Applying Voronoi summation}

We substitute \eqref{finalchar} back into \eqref{transformed} getting
\begin{equation}\label{new}
\begin{split}
TK \sum_{\substack{M_1, M_2, G\geq 1\\ \text{powers of two}}}   &  \sum_{\pm} \frac{1\pm i}{2}\sum_{\substack{(\delta,N)=1\\ (d_1, d_2) = 1}} \frac{\chi(d_1)\ov{\chi(d_2)}}{\delta^2d_1d_2}\sum_{N \mid c}\frac{\phi(c)}{\phi(d_1d_2c)}\frac{\omega(c/G)}{c^{1/2}}\\
&\times\sum_{\substack{gh = c\\(g, d_1d_2) = 1}}
\asterisk\sum_{\substack{f\, (d_1^2d_2h)\\ (g\bar{f} \mp d_2, d_1^2d_2h) = d_1}} \chi\left( \frac{g\bar{f}\mp d_2}{d_1}\right) \mathcal{S}^{M_1, M_2}_{d_1, d_2, c}(f, g, h),
\end{split}
\end{equation}
where
\begin{displaymath}
\begin{split}
 \mathcal{S}^{M_1, M_2}_{d_1, d_2, c}(f, g, h) & := \sum_{m_1, m_2}    \lambda_g(m_1)  \ov{\lambda_g(m_2) } e\left(\pm m_1\frac{\ov{(g \mp d_2f)/d_1} }{d_2h} + m_2 \frac{\bar{f}}{d_1h} \right) F^{M_1, M_2}_{d_1, d_2, c}(m_1, m_2),\\
 F^{M_1, M_2}_{d_1, d_2, c}(x, y) & :=  \frac{\omega(x/M_1)\omega(y/M_2)}{(xy)^{3/4}}\Omega_3\left(T \log \frac{d_2^2y}{d_1^2x}, \frac{\pm K^2c}{8\pi \sqrt{xy}}, \frac{\delta^4d_1^2d_2^2xy}{\sqrt{\tilde{C}}}\right)e\left(\mp \frac{(\sqrt{d_1^2x} - \sqrt{d_2^2y})^2}{d_1d_2c}\right).
\end{split}
\end{displaymath}
By applying Proposition~\ref{voro} for the summation variables $m_1$ and $m_2$, we see that \[d_1d_2h^2\mathcal{S}^{M_1, M_2}_{d_1, d_2, c}(f, g, h)\] is a sum of terms (suppressing $M_1$ and $M_2$ from the notation for simplicity)
\begin{equation}\label{vormain}\sum_{m_1, m_2} \lambda_g(m_1)\ov{\lambda_g(m_2)}e\left(\mp m_1\frac{g}{d_1d_2h}\mp(m_2\mp m_1)\frac{f}{d_1 h}\right)F_{d_1,d_2,c}^{\pm,\pm}(m_1,m_2)\end{equation}
with
\begin{equation}\label{Fdef2}F_{d_1,d_2,c}^{\pm,\pm}(m_1,m_2):=\int_0^\infty\int_0^\infty F^{M_1, M_2}_{d_1,d_2,c}(x,y)\,
J_g^\pm\left(\frac{4\pi\sqrt{m_1x}}{d_2h}\right)J_g^\pm\left(\frac{4\pi\sqrt{m_2y}}{d_1h}\right)dx\,dy.\end{equation}
If $g = E_r$ is an Eisenstein series, there are three additional polar terms
\begin{align}
\label{newpolar1}&\sum_{m_1}\lambda_g(m_1)e\left(\mp m_1\frac{g\mp d_2f}{d_1d_2h}\right)\int_0^\infty F_{d_1, d_2, c}^{\pm, 0}(m_1, y) \,P_{r,d_1h}^\pm(y)\,dy,\\
\label{newpolar2}&\sum_{m_2}\ov{\lambda(m_2)}e\left(\mp m_2\frac{f}{d_1h}\right)\int_0^\infty F_{d_1,d_2,c}^{0,\pm}(x,m_2)\,P_{r,d_2h}^\pm(x)\,dx,\\
\label{newpolar3}&\int_0^\infty\int_0^\infty F^{M_1, M_2}_{d_1, d_2, c}(x,y)\,P_{r,d_2h}^\pm(x)\,P_{r,d_1h}^\pm(y)\,dx\,dy,
\end{align}
with
\begin{align}
\label{Fplus} F_{d_1,d_2,c}^{\pm,0}(m_1,y)&:=\int_0^\infty F^{M_1, M_2}_{d_1,d_2,c}(x,y)\,J_g^\pm\left(\frac{4\pi\sqrt{m_1x}}{d_2h}\right)dx,\\
\label{Fminus} F_{d_1,d_2,c}^{0,\pm}(x,m_2)&:=\int_0^\infty F^{M_1, M_2}_{d_1,d_2,c}(x,y)\,J_g^\pm\left(\frac{4\pi\sqrt{m_2y}}{d_1h}\right)dy,\\[4pt]
\label{Pdefinition}P_{r,c}^\pm(t)&:=\begin{cases}
\zeta(1\pm 2ir)(t/c^2)^{\pm ir},&\text{for $r\neq 0$},\\
\log(t/c^2)+2\gamma,&\text{for $r=0$}.
\end{cases}
\end{align}
We proceed to analyze the four terms \eqref{vormain}, \eqref{newpolar1}, \eqref{newpolar2}, \eqref{newpolar3}. It will turn out that the first $3$ terms are small, but \eqref{newpolar3} contributes to the main term.

\subsection{The contribution of \eqref{vormain}}
We observe first that in \eqref{Fdef2} the arguments of the Bessel functions are large:
\begin{displaymath}
  \frac{4\pi\sqrt{m_1x}}{d_2h} \gg \frac{\sqrt{M_1}}{d_2G} \gg \frac{\sqrt{M_1}K^2}{\tilde C^{\eps} d_2 \sqrt{M_1M_2}} \asymp \frac{K^2}{\tilde C^{\eps}(d_1d_2)^{1/2} (M_1M_2)^{1/4}} \gg \frac{K^2}{\tilde C^{\eps}(T+K)}
\end{displaymath}
by \eqref{sizes}, \eqref{sizeM} and the fact that $h \leq c$. A similar estimate holds  for $4\pi \sqrt{m_2y}/(d_1h)$. In view of \eqref{assumption} this is large. In particular, by the rapid decay of the Bessel $K$-function \eqref{besselK}, among $F_{d_1,d_2,c}^{\pm,\pm}(m_1,m_2)$ we only need to consider $F_{d_1,d_2,c}^{+,+}(m_1,m_2)$ as the contribution of the other $3$ expressions is negligible (or zero).

We show now that the sum \eqref{vormain} can be truncated efficiently.
Fix any $y \in [(1/2) M_2, (5/2)M_2]$ in the integral defining
$F_{d_1,d_2,c}^{+,+}(m_1,m_2)$. By \eqref{close}, we can restrict the $x$-integration to
\begin{equation}\label{close2}
d_1^2 x=d_2^2y\bigl(1+O(T^{\eps-1})\bigr)
\end{equation} at the cost of a negligible error. In this range we have
\[\frac{\partial^j}{\partial x^j}F_{d_1, d_2, c}(x, y)\ll_{j}\tilde C^\eps\left(\frac{T}{M_1} + \frac{d_1}{Td_2c}\right)^j,\]
so that by Lemma~\ref{lemma4} the integral is negligible unless
\begin{displaymath}
  \left(\frac{T}{M_1} + \frac{d_1}{Td_2c}\right)\frac{\sqrt{M_1}d_2h}{\sqrt{m_1}} \geq \tilde C^{-\eps},
\end{displaymath}
that is
\begin{equation}\label{m1}
  m_1 \leq  M^{\ast}_1 := \tilde C^{\eps} d_2^2\left(\frac{(T h)^2}{M_1} + \frac{M_2 }{(Tg)^2}\right).
\end{equation}
Here we used \eqref{sizeM} and the fact that $gh=c$. Similarly, we can assume
\begin{equation}\label{m2}
   m_2 \leq M_2^{\ast} := \tilde C^{\eps} d_1^2\left(\frac{(Th )^2}{M_2} + \frac{M_1 }{(Tg)^2}\right).
\end{equation}
For convenience we observe, by \eqref{sizeM},
\begin{equation}\label{mstar}
  M^\ast:=\max(M_1^{\ast},M_2^{\ast})\asymp \tilde C^{\eps} d_1d_2
  \left(\frac{(Th)^2}{(M_1M_2)^{1/2}}+\frac{(M_1M_2)^{1/2}}{(Tg)^2}\right).
\end{equation}
To summarize,
\begin{equation}\label{almostfinal}
\begin{split}
  \mathcal{S}^{M_1, M_2}_{d_1, d_2, c}(f, g, h) = \frac{1}{d_1d_2h^2} & \sum_{m_1 \leq M^{\ast}} \lambda_g(m_1) e\left(\mp \frac{m_1g}{d_1d_2h}\right) \sum_{m_2 \leq M^{\ast}} \ov{\lambda_g(m_2)}  e\left(( m_1 - m_2)\frac{f}{d_1h} \right) \\ & \times \int_0^{\infty}\int_0^{\infty} F^{M_1, M_2}_{d_1, d_2, c}(x, y) J_g^{+}\left(\frac{4\pi \sqrt{m_1x}}{d_2h}\right)J_g^{+}\left(\frac{4\pi \sqrt{m_2y}}{d_1h}\right) dx\,dy,
  \end{split}
\end{equation}
up to negligible error terms and the contribution of the three polar terms \eqref{newpolar1}--\eqref{newpolar3} that we discuss in a moment.

If we substitute this back into \eqref{new}, then the summation over $f$ produces
the exponential sum
\begin{displaymath}
   \asterisk\sum_{\substack{f\, (d_1^2d_2h)\\ (g \mp d_2f, d_1^2d_2h) = d_1}} \chi\left(\frac{g\bar{f} \mp d_2}{d_1}\right) e\left( (m_1 - m_2)\frac{f}{d_1h} \right).
\end{displaymath}
We estimate the double integral in \eqref{almostfinal} using \eqref{sizes}, \eqref{close2}, and \eqref{Bessellarge} as
\[\ll (TM_1M_2)^{\eps-1}\min\left(\frac{d_1^2M_1^2}{d_2^2},\frac{d_2^2M_2^2}{d_1^2}\right)
\frac{(d_1d_2)^{1/2}h}{(m_1m_2)^{1/4}}\ll \tilde C^{\eps} \frac{(d_1d_2)^{1/2}h}{T(m_1m_2)^{1/4}}.\]
By Lemma \ref{lem6} and the support of $\omega$, the contribution of \eqref{vormain} to \eqref{new} is
\begin{equation*}
\begin{split}
&\ll TK \tilde C^{\eps}   \sum_{  \delta^2d_1d_2 \leq \tilde C^{1/4+\eps}  } \frac{1}{\delta^2(d_1d_2)^{3/2}}
\sum_{\substack{gh \leq 3G\\m_1,m_2 \leq M^{\ast}}} \frac{(h, (m_1-m_2)d_1d_2)}{(gh)^{1/2}hT}\frac{|\lambda_g(m_1)\lambda_g(m_2)|}{(m_1m_2)^{1/4}} \\
& \ll TK \tilde C^{\eps}   \sum_{  \delta^2d_1d_2 \leq \tilde C^{1/4+\eps}  } \frac{1}{\delta^2(d_1d_2)^{3/2}}   \sum_{gh \leq 3G}\frac{1}{(gh)^{1/2}hT} \sum_{\ell\mid h}\ell\sum_{\substack{ m_1, m_2 \leq M^{\ast}\\ \ell\mid(m_1-m_2)d_1d_2}}  \frac{|\lambda_g(m_1)\lambda_g(m_2)|}{(m_1m_2)^{1/4}}.
\end{split}
\end{equation*}
By \eqref{boundfourier} the innermost sum is
\[\leq\sum_{\substack{ m_1, m_2 \leq M^{\ast}\\ \ell\mid(m_1-m_2)d_1d_2}}  \frac{1}{2}\left(\frac{|\lambda_g(m_1)|^2}{m_1^{1/2}}+\frac{|\lambda_g(m_2)|^2}{m_2^{1/2}}\right)
\ll (M^{\ast})^{1/2+\eps}\left(1+\frac{M^\ast(\ell,d_1d_2)}{\ell}\right),\]
hence in the end the contribution of \eqref{vormain} to \eqref{new} is, using also \eqref{sizes} and \eqref{mstar},
\begin{equation}\label{offerror2}
\begin{split}
&\ll TK \tilde C^{\eps} \sum_{  \delta^2d_1d_2 \leq \tilde C^{1/4+\eps}  } \frac{1}{\delta^2(d_1d_2)^{3/2}}   \sum_{gh \leq 3G}\frac{1}{(gh)^{1/2}hT} \left(h(M^{\ast})^{1/2} + (h,d_1d_2)(M^{\ast})^{3/2}\right)\\
&\ll TK \tilde C^{\eps} \sum_{  \delta^2d_1d_2 \leq \tilde C^{1/4+\eps}  } \frac{1}{\delta^2}  \left(\frac{G^{3/2}}{(d_1d_2)(M_1M_2)^{1/4}} + \frac{G^{1/2}(M_1M_2)^{1/4}}{(d_1d_2)T^2} + \frac{T^2G^{5/2}}{(M_1M_2)^{3/4}} + \frac{(M_1M_2)^{3/4}}{T^4}\right)\\
&\ll TK \tilde C^{\eps} \sum_{  \delta^2d_1d_2 \leq \tilde C^{1/4+\eps}  } \frac{1}{\delta^2} \left(\frac{(M_1M_2)^{1/2}}{(d_1d_2)K^3} + \frac{(M_1M_2)^{1/2}}{(d_1d_2)T^2K} + \frac{T^2(M_1M_2)^{1/2}}{K^5} + \frac{(M_1M_2)^{3/4}}{T^4} \right)\\
&\ll TK \tilde C^{\eps} \sum_{  \delta^2d_1d_2 \leq \tilde C^{1/4+\eps}  } \frac{1}{\delta^2d_1d_2} \left(\frac{\tilde C^{1/4}}{K^3} + \frac{\tilde C^{1/4}}{T^2K} + \frac{T^2\tilde C^{1/4}}{K^5} + \frac{\tilde C^{3/8}}{T^4} \right)\\
&\ll (TK)^{1+\eps} \left(\frac{T^4}{K^5} + \frac{K^3}{T^4}\right).
\end{split}
\end{equation}

\subsection{The contribution of \eqref{newpolar1} and \eqref{newpolar2}}

We show that the integrals \eqref{Fplus}--\eqref{Fminus} are negligible in the ranges \eqref{m1}--\eqref{m2}.  Starting with the first, we see by the rapid decay of $\Omega_3$ in the definition of $F_{d_1,d_2,c}(x,y)$ that the contribution of $|d_1^2 x-d_2^2 y|\geq  T^{\eps-1} d_1^2 M_1$ is negligible. Introducing $w:=(\sqrt{d_1^2 x}-\sqrt{d_2^2 y})^2$ and applying several integrations by parts with respect to $w$ shows that the contribution of $\sqrt{wd_1^2 x}\geq d_1d_2c T^{1+\eps}$ is also negligible. Using
$|d_1^2 x-d_2^2 y|\ll\sqrt{wd_1^2 x}$ we infer that we can restrict the integration to
\[|d_1^2 x-d_2^2 y|\leq Z:=\tilde{C}^\eps\min\left(\frac{d_1^2M_1}{T}, d_1d_2cT\right)\]
at the cost of a negligible error. In other words, writing $z:=d_1^2x-d_2^2y$ we can approximate the above integral by
\[-\frac{1}{d_2^2}\int_{-Z}^Z\int_0^\infty F_{d_1,d_2,c}\left(x,\frac{d_1^2x-z}{d_2^2}\right)P^\pm_{r,d_1h}\left(\frac{d_1^2x-z}{d_2^2}\right)
\,J_g^\pm\left(\frac{4\pi\sqrt{m_1x}}{d_2h}\right)dx\,dz\]
with negligible error. For $|z|\leq Z$ we have
\begin{align*}
\frac{\partial^j}{\partial x^j}\left\{
F_{d_1,d_2,c}\left(x,\frac{d_1^2x-z}{d_2^2}\right)P^\pm_{r,d_1h}\left(\frac{d_1^2x-z}{d_2^2}\right)
\right\}
&\ll_{j}\tilde{C}^\eps\left(\frac{K|z|}{d_1^2M_1^2}+\frac{1}{M_1}+\frac{z^2}{(d_1^2M_1^2)(d_1d_2c)}\right)^j\\[4pt]
&\ll_{j}\tilde{C}^\eps M_1^{-j},
\end{align*}
whence by Lemma~\ref{lemma4} and \eqref{sizes} the last integral is
\begin{align*}
&\ll_{j}\tilde{C}^\eps Z\left(\frac{d_2h}{\sqrt{M_1}}\right)^j
\ll_{j}\tilde{C}^\eps Z\left(\frac{(d_1d_1)^{1/2}G}{(M_1M_2)^{1/4}}\right)^j
\ll_{j}\tilde{C}^\eps Z\left(\frac{(d_1d_1)^{1/2}(M_1M_2)^{1/4}}{K^2}\right)^j\\
& \ll_{j} \tilde{C}^\eps Z \left(\frac{\tilde{C}^{1/8}}{K^2}\right)^j \ll \tilde{C}^{\varepsilon} Z \left(\frac{T+K}{K^2}\right)^j.
\end{align*}
Choosing $j\in\NN$ sufficiently large and using \eqref{assumption}, \eqref{sizes}, \eqref{m1} we conclude that the polar term \eqref{newpolar1} is negligible. In the same way we see that the polar term \eqref{newpolar2} is negligible.

\subsection{The final polar term}\label{secpolar}

It remains to estimate the contribution of the polar term \eqref{newpolar3}. This requires some non-trivial manipulation. First we remove the  dyadic decompositions by summing over $M_1, M_2, G$ using \eqref{dyadic}, but we keep in mind that the decay of $\Omega_3$ allows us to restrict to $\delta, d_1, d_2, c \leq K^{\eps}$, up to a negligible error.  This gives
\begin{align}\label{finalpolar}
&    TK  \sum_{\pm} \frac{1\pm i}{2}  \sum_{\substack{(\delta,N)=1 \\ (d_1, d_2) = 1}} \frac{\chi(d_1)\ov{\chi(d_2)}}{\delta^2d_1d_2}
    \sum_{N \mid c}   \frac{\phi(c) }{\phi(d_1d_2c)}\frac{1}{c^{1/2}}\sum_{\substack{gh = c\\(g, d_1d_2) = 1}}
    \asterisk\sum_{\substack{f\, (d_1^2d_2h)\\ (g\bar{f} \mp d_2, d_1^2d_2h) = d_1}} \frac{1}{d_1d_2h^2} \chi\left( \frac{g\bar{f}\mp d_2}{d_1}\right)\\
\nonumber
& \times\int_0^\infty\int_0^\infty \Omega_3\left(T \log \frac{d_2^2y}{d_1^2x}, \frac{\pm K^2c}{8\pi \sqrt{xy}}, \frac{\delta^4 d_1^2d_2^2xy}{\sqrt{\tilde{C}}}\right)    e\left(\mp \frac{(\sqrt{d_1^2x} - \sqrt{d_2^2y})^2}{d_1d_2c}\right)  \sum_{\pm, \pm} P_{r,d_2h}^\pm(x) P_{r,d_1h}^\pm(y) \frac{dx\,dy}{(xy)^{3/4}},
\end{align}
up to a negligible error coming from the fact that $\sum_{j\geq 0} \omega(x/2^j)$ is not necessarily $1$ for $x < 1$. In order to simplify the notation, we will only consider the $(+, +)$-term in the last sum and drop the superscripts at $P$.

We observe that the $\delta,d_1, d_2$-sum is rapidly converging due to the decay properties of $\Omega_3$. A trivial estimation shows that the double integral is $O(\tilde{C}^{\eps})$.  Now we replace $\Omega_3$ by $\tilde{\Omega}_3$ which introduces by \eqref{tilde} an admissible error of
\begin{equation}\label{errorpolar1}
 O\left((TK)^{1+\varepsilon}(T+K)^{-1}\right).
\end{equation}
Next we make a change of variables
\begin{displaymath}
  z := xy, \qquad w := \frac{d_2^2y}{d_1^2 x} - 1.
\end{displaymath}
By Taylor's formula,
\[T \log \frac{d_2^2y}{d_1^2x} = Tw + O(Tw^2) = Tw + O(T^{\eps-1})\]
in the range where $\tilde{\Omega}_3$ is not negligible. Again by Taylor's formula,
\begin{displaymath}
 \left(\sqrt{d_1^2x} - \sqrt{d_2^2y}\right)^2 = \frac{1}{4}d_1d_2 \sqrt{z}  w^2 (1+O(w)),
\end{displaymath}
hence
\begin{displaymath}
  e\left(\mp \frac{(\sqrt{d_1^2x} - \sqrt{d_2^2y})^2}{d_1d_2c}\right) = e\left(\mp \frac{\sqrt{z} w^2 }{4c}\right)\bigl(1+O(T^{\eps-1})\bigr)
\end{displaymath}
in the range where $\tilde{\Omega}_3$ is not negligible. Therefore the double integral equals
\begin{align}
\nonumber\int_{-1}^\infty\int_0^\infty &\tilde{\Omega}_3\left(T w, \frac{\pm K^2c}{8\pi \sqrt{z}}, \frac{\delta^4d_1^2d_2^2z}{\sqrt{\tilde{C}}}\right)    e\left(\mp \frac{\sqrt{z}w^2}{4c}\right) \\\times
\nonumber&P_{r,d_2h}\left(\frac{d_2\sqrt{z}}{d_1\sqrt{1+w}}\right)
   P_{r,d_1h}\left(\frac{d_1\sqrt{z(1+w)}}{d_2}\right) \frac{dw\,dz}{2(1+w)z^{3/4}}  + O(T^{\eps-1})\\
\label{polarerror2}=  \int_{-\infty}^\infty\int_0^\infty &\tilde{\Omega}_3\left(T w, \frac{\pm K^2c}{8\pi \sqrt{z}}, \frac{\delta^4d_1^2d_2^2z}{\sqrt{\tilde{C}}}\right)    e\left(\mp \frac{\sqrt{z}w^2}{4c}\right) \\\times
\nonumber &P_{r,d_2h}\left(\frac{d_2\sqrt{z}}{d_1}\right)
   P_{r,d_1h}\left(\frac{d_1\sqrt{z}}{d_2}\right) dw\,\frac{dz}{2z^{3/4}}  + O(T^{\eps-1}).
\end{align}

We rewrite the last $w$-integral using the definition \eqref{defom3}:
\begin{equation}\label{triple}
I:=\int_{-\infty}^{\infty} \int_0^{\infty}
\tilde{\Omega}_2\left(x, \frac{\pm K^2c}{8\pi \sqrt{z}}, \frac{\delta^4d_1^2d_2^2z}{\sqrt{\tilde{C}}}\right)
\exp\left(iTwx \mp \frac{i\pi\sqrt{z} w^2}{2c}\right)\,dx\,dw.
\end{equation}
This double integral is not absolutely convergent, so we consider
\begin{equation}\label{tripleeps}
I_\eps:=\int_{-\infty}^{\infty} \int_0^{\infty}
\tilde{\Omega}_2\left(x, \frac{\pm K^2c}{8\pi \sqrt{z}}, \frac{\delta^4d_1^2d_2^2z}{\sqrt{\tilde{C}}}\right)
\exp\left(iTwx - \frac{e^{\pm i(\pi/2-\eps)} \pi\sqrt{z} w^2}{2c}\right)\,dx\,dw.
\end{equation}
for $\varepsilon > 0$. The $x$-integral in \eqref{triple}--\eqref{tripleeps}
is a Fourier transform decaying rapidly in $w$,
therefore by cutting the $w$-integral at larger and larger parameters we see that $\limsup_{\eps\to 0+}|I-I_\eps|$
is smaller than any positive number, i.e.\ $I=\lim_{\eps\to0+}I_\eps$.
As the double integral \eqref{tripleeps} is absolutely convergent, we can change the order of integration there and compute the $w$-integral using \cite[3.323.2]{GR}:
\begin{displaymath}
  I_{\varepsilon} =  e^{\mp i(\frac{\pi}{4} - \frac{\varepsilon}{2})} \left(\frac{2c}{\sqrt{z}}\right)^{1/2}   \int_0^{\infty} \tilde{\Omega}_2\left(x, \frac{\pm K^2c}{8\pi \sqrt{z}}, \frac{\delta^4d_1^2d_2^2z}{\sqrt{\tilde{C}}}\right) \exp\left(-\frac{e^{\mp i(\pi/2 - \varepsilon)}cT^2x^2}{2\pi \sqrt{z}}\right) \,dx.
\end{displaymath}
Here the integrand is rapidly decaying in $x$, hence by a limsup argument as before we see that
\begin{displaymath}
I=\lim_{\eps\to 0+}I_\eps= e^{\mp i\frac{\pi}{4}} \left(\frac{2c}{\sqrt{z}}\right)^{1/2}   \int_0^{\infty} \tilde{\Omega}_2\left(x, \frac{\pm K^2c}{8\pi \sqrt{z}}, \frac{\delta^4d_1^2d_2^2z}{\sqrt{\tilde{C}}}\right) \exp\left(\frac{\pm icT^2x^2}{2\pi \sqrt{z}}\right) \,dx.
\end{displaymath}
Using also the definition \eqref{defom2} we can summarize that the $w$-integral in \eqref{polarerror2} equals
\[(1\mp i)\left(\frac{c}{\sqrt{z}}\right)^{1/2}
\int_0^{\infty}\int_0^{\infty} \tilde{\Omega}_1\left(x, y, \frac{\delta^4d_1^2d_2^2z}{\sqrt{\tilde{C}}}\right) \exp\left(\pm \frac{ic}{2\pi\sqrt{z}} \left(x^2T^2 + \frac{y^2K^2}{4}\right)\right) \,dy\,dx.\]

We integrate this over $z$ and substitute it back into \eqref{finalpolar} getting
\begin{displaymath}
\begin{split}
&TK  \sum_{\pm}  \sum_{\substack{(\delta,N)=1 \\ (d_1, d_2) = 1}} \frac{\chi(d_1)\ov{\chi(d_2)}}{(\delta d_1d_2)^2}
\sum_{N \mid c}   \frac{\phi(c) }{\phi(d_1d_2c)} \sum_{\substack{gh = c\\(g, d_1d_2) = 1}} \frac{1}{h^2}
\asterisk\sum_{\substack{f\, (d_1^2d_2h)\\ (g\bar{f} \mp d_2, d_1^2d_2h) = d_1}}  \chi\left( \frac{g\bar{f}\mp d_2}{d_1}\right)\\
\nonumber
&\times\int_0^{\infty} \int_0^{\infty}\int_0^{\infty} \tilde{\Omega}_1\left(x, y, \frac{\delta^4d_1^2d_2^2z}{\sqrt{\tilde{C}}}\right) \exp\left(\pm \frac{ic}{2\pi\sqrt{z}} \left(x^2T^2 + \frac{y^2K^2}{4}\right)\right) \\
& \qquad\qquad\qquad \times P_{r,d_2h}\left(\frac{d_2\sqrt{z}}{d_1}\right)
   P_{r,d_1h}\left(\frac{d_1\sqrt{z}}{d_2}\right) dy\,dx \,\frac{dz}{2z}.
   \end{split}
\end{displaymath}
Let us consider the case $r \neq 0$ and insert the $+$ case of  \eqref{Pdefinition}. Then we can recast the preceding display as
\begin{equation}\label{a}
\begin{split}
&   \zeta(1+2ir)^2 \,TK  \sum_{\pm}  \sum_{\substack{(\delta,N)=1 \\ (d_1, d_2) = 1}} \frac{\chi(d_1)\ov{\chi(d_2)}}{(\delta d_1d_2)^2}
    \sum_{N \mid c}   \frac{\phi(c) }{\phi(d_1d_2c)} \sum_{\substack{gh = c\\(g, d_1d_2) = 1}} \frac{1}{h^2}
    \asterisk\sum_{\substack{f\, (d_1^2d_2h)\\ (g\bar{f} \mp d_2, d_1^2d_2h) = d_1}}   \chi\left( \frac{g\bar{f}\mp d_2}{d_1}\right)\\
& \times  \int_0^{\infty} \int_0^{\infty}\int_0^{\infty} \tilde{\Omega}_1\left(x, y, \frac{\delta^4d_1^2d_2^2z}{\sqrt{\tilde{C}}}\right) \exp\left(\pm \frac{ic}{2\pi\sqrt{z}} \left(x^2T^2 + \frac{y^2K^2}{4}\right)\right)
\left(\frac{\sqrt{z}}{d_1d_2h^2}\right)^{2ir} dy\,dx \,\frac{dz}{2z}.
   \end{split}
\end{equation}
We make a change of variables
\begin{displaymath}
  \frac{c}{\pi \sqrt{z}} = v, \qquad z = \frac{c^2}{\pi^2 v^2},
\end{displaymath}
and write the triple integral as
\begin{displaymath}
  \int_0^{\infty} \int_0^{\infty}\int_0^{\infty} \tilde{\Omega}_1\left(x, y, \frac{\delta^4d_1^2d_2^2c^2}{\pi^2v^2\sqrt{\tilde{C}}}\right) \exp\left(\pm \frac{1}{2} iv \left((xT)^2+(\tfrac{1}{2}yK)^2\right)\right)  \left(\frac{c}{\pi v d_1d_2h^2}\right)^{2ir} dy\,dx \,\frac{dv}{v}.
\end{displaymath}

Finally  we insert the definition \eqref{defom1} and arrive at
\begin{displaymath}
\begin{split}
 & \int_0^{\infty} \int_0^{\infty}\int_0^{\infty} W_1(x)W_2(y) \frac{1}{2\pi i} \int_{(2)} \widecheck{G}_0 (s) \left(   \frac{4 \pi  \delta^2d_1d_2c}{ v N \left((xT)^2+(\frac{1}{2}yK)^2\right)}\right)^{-2s} ds  \\
 &\quad\quad\quad  \times \exp\left(\pm \frac{1}{2} iv \left((xT)^2+(\tfrac{1}{2}yK)^2\right)\right)  \left(\frac{c}{\pi v d_1d_2h^2}\right)^{2ir} dy\,dx \,\frac{dv}{v} \\
  \end{split}
\end{displaymath}
by Mellin inversion. We compute the $v$-integral by \cite[3.381.4]{GR}; the change of the order of integration can be justified similarly as before. Reorganizing, we obtain
\begin{align*}
&\left(\frac{c}{2\pi  d_1d_2h^2}\right)^{2ir} \int_0^{\infty}\int_0^{\infty} W_1(x)W_2(y) \left((xT)^2+(\tfrac{1}{2}yK)^2\right)^{2ir} \,dy\,dx\\
&\times \frac{1}{2\pi i} \int_{(2)} \widecheck{G}_0 (s) \left(   \frac{N}{2 \pi  \delta^2d_1d_2c}\right)^{2s}\Gamma(2s-2ir) \exp(\pm i \pi(s-ir)) \,ds,
\end{align*}
where by \eqref{smoothlog} the $x, y$-integral is simply
\begin{displaymath}
  \mathcal{M}_{ir}(T, K) \left(\frac{(2\pi)^4}{N^2}\right)^{ir}.
\end{displaymath}
We substitute this into \eqref{a}, recall the definition \eqref{defXi} and recast \eqref{a} as
\begin{displaymath}
\zeta(1+2ir)^2 \,TK \mathcal{M}_{ir}(T, K) \,\frac{1}{2\pi i} \int_{(2)} \widecheck{G}_0 (s) \,\Xi_r(s) \,N^{2s-2ir}
2(2\pi)^{-2s+2ir} \Gamma(2s-2ir) \cos(  \pi(s-ir)) \, ds,
\end{displaymath}
where $\Xi_r(s)=\Xi_r^\pm(s)$ denotes the function in \eqref{finalXi}.

To summarize, the contribution of \eqref{newpolar3} equals, up to an admissible error,
\begin{equation}\label{finalcontr}\zeta(1+2ir)^2 \prod_{p \mid N} (1-p^{-2-4ir}) \,TK \mathcal{M}_{ir}(T, K) \,\frac{1}{2\pi i}
\int_{(2)} \widecheck{G}_0 (s) \,Z_{r}(s)\,ds\end{equation}
with the kernel
\[Z_{r}(s):=2(2\pi)^{-2s+2ir}\Gamma(2s-2ir) \cos(  \pi(s-ir))\zeta(2s-2ir)\zeta(1+2s+2ir).\]
Using \cite[8.334.2 \& 8.335.1]{GR} the kernel equals
\[Z_{r}(s)=\pi^{\frac{1}{2}-2s+2ir}\frac{\Gamma(s-ir)}{\Gamma(\frac{1}{2}-s+ir)}\zeta(2s-2ir)\zeta(1+2s+2ir),\]
hence by the functional equation for the Riemann zeta function we obtain
\[Z_{r}(s)=\zeta(1-2s+2ir)\zeta(1+2s+2ir).\]

Let us assume $r\neq 0$, then the integrand in \eqref{finalcontr} is an odd function of $s$ which is holomorphic except for a simple pole
at $s=0$ as well as possible poles at   $s = \pm ir$. 
It follows that the $s$-integral equals half the sum of its residues, that is, \eqref{finalcontr} equals
\begin{equation}\label{constbpm}
\begin{split}
& \frac{1}{2}  \zeta(1+2ir)^4 \prod_{p \mid N}(1-p^{-2-4ir})\,TK \mathcal{M}_{ir}(T,K)\\
- & \frac{1}{2} \zeta(1+2ir)^2 \zeta(1+4ir)  \prod_{p \mid N}(1-p^{-2-4ir})\, \widecheck{G}_0(ir) \, TK\mathcal{M}_{ir}(T, K).
\end{split}
\end{equation}
Here we also used that $\widecheck{G}_0(-ir) = - \widecheck{G}_0(ir)$. The $(-, -)$ case in the last sum of \eqref{finalpolar} gives a similar contribution to $\mathcal{M}_{-ir}(T, K)$, so that the second line of \eqref{constbpm} together with the corresponding part of the $(-, -)$ case precisely cancels \eqref{extrares1}. The $(+, -)$ and $(-, +)$ cases contribute to $\mathcal{L}_0$.

The case $r=0$ can be treated similarly and gives a linear combination of $\mathcal{L}_0$, $\mathcal{L}_1$, $\mathcal{L}_2$ in Theorem~\ref{theorem3}.

This completes the discussion of the off-diagonal term. The various error terms \eqref{offerror1}, \eqref{offerror2}, \eqref{errorpolar1}, \eqref{polarerror2} encountered so far are admissible for Theorems~\ref{theorem1}--\ref{theorem3}, while the first line of \eqref{constbpm}
and its counterpart with $ir$ replaced by $-ir$ have the desired shape for Theorem \ref{theorem2}. The proofs are complete.

\section{Appendix}

In this Appendix we deduce Lemma~\ref{Dlemma} from the modularity of $g=E_r$ $(r\neq 0)$,
in order to emphasize the analogy with the cuspidal case. We follow closely \cite[Section~2.4]{HM}.

By \cite[(3.29)]{Iw3} we have the Fourier decomposition
\begin{align*}
\theta\left(\tfrac{1}{2}+ir\right)g(x+iy)\,=\
&\theta\left(\tfrac{1}{2}+ir\right)y^{\frac{1}{2}+ir}+\theta\left(\tfrac{1}{2}-ir\right)y^{\frac{1}{2}-ir}\\
+\ &4\sqrt{y}\sum_{n=1}^\infty\lambda_g(n)K_{ir}(2\pi ny)\cos(2\pi nx),
\end{align*}
where
\begin{equation}\label{kz}
\theta(z):=\pi^{-z}\Gamma(z)\zeta(2z)
\qquad\text{and}\qquad
\lambda_g(n):=\sum_{ab=n}\left(\frac{a}{b}\right)^{ir}.
\end{equation}
For convenience we introduce
\[D^{\pm 1}(g,x,s):=\frac{1}{2}D(g,x,s)\pm\frac{1}{2}D(g,-x,s),\]
i.e.\
\begin{align*}
D^{+1}(g,x,s)&=\sum_{n=1}^\infty \lambda_g(n) \cos(2\pi nx) n^{-s},\\
D^{-1}(g,x,s)&=\sum_{n=1}^\infty \lambda_g(n) i\sin(2\pi nx) n^{-s}.
\end{align*}
It will be more pleasant to work with the Maass shift (cf.\ \cite[(4.3)]{DFI})
\[\tilde g(x+iy):=-y\left(i\frac{\partial}{\partial x}+\frac{\partial}{\partial y}\right)\theta\left(\tfrac{1}{2}+ir\right)g(x+iy)\]
which is a weight $2$ Eisenstein series of Laplacian eigenvalue $1/4+r^2$.
By \cite[Section~4]{DFI} we have the Fourier decomposition
\[\tilde g(x+iy)=\gconst(y)+\gser(x+iy),\]
where
\begin{align}
\label{gconst}\gconst(y)&:=
-\left(\tfrac{1}{2}+ ir\right)\theta\left(\tfrac{1}{2}+ir\right)y^{\frac{1}{2}+ ir}
-\left(\tfrac{1}{2}- ir\right)\theta\left(\tfrac{1}{2}-ir\right)y^{\frac{1}{2}- ir},\\
\label{gser}\gser(x+iy)&:=\sum_{n=1}^\infty\frac{\lambda_g(n)}{\sqrt{n}}
\left\{V_{2,ir}^{+1}(\pi ny)\cos(2\pi nx)+V_{2,ir}^{-1}(\pi ny)i\sin(2\pi nx)\right\},
\end{align}
and $V_{2,ir}^{\pm 1}$ is as in \cite[(8.27)]{DFI}
\[V_{2,ir}^{\pm 1}(y):=W_{1,ir}(4y)\mp\left(\tfrac{1}{4}+r^2\right)W_{-1,ir}(4y).\]
Utilizing the functional equation
\begin{equation}\label{modular}
\tilde g\left(\frac{a}{c}+\frac{iy}{c}\right)=-\tilde g\left(-\frac{\ov a}{c}+\frac{i}{cy}\right),\qquad y>0,
\end{equation}
which is a consequence of modularity, we see by standard estimates that
\begin{equation}\label{gserbound}
\gser\left(\frac{a}{c}+\frac{iy}{c}\right)\ll_{g,c}\min(y^{-1/2},y e^{-2\pi y}).\end{equation}
Therefore, taking Mellin transforms we obtain by \eqref{gser}
\begin{equation}\label{intrep}\int_0^\infty \gser\left(\frac{a}{c}+\frac{iy}{c}\right)
y^{s-\frac{1}{2}}\frac{dy}{y}=
\frac{1}{\sqrt{c}}\left(\frac{c}{\pi}\right)^s F\left(g,\frac{a}{c},s\right),\qquad\Re s>1,\end{equation}
where
\begin{equation}\label{Fdef}
F(g,x,s):=\Phi_2^{+1}(s,ir)D^{+1}(g,x,s)+\Phi_2^{-1}(s,ir)D^{-1}(g,x,s)\end{equation}
and
\[\Phi_2^{\pm 1}(s,ir):=\sqrt{\pi}\int_0^\infty V_{2,ir}^{\pm 1}(y)\,y^{s-\frac{1}{2}}\frac{dy}{y}.\]
The last function is the same as \cite[(8.25)]{DFI} except that we deleted the $4$ from the denominator.
The reason is that we would like to apply \cite[Lemma~8.2]{DFI} but that lemma requires this
correction, because the initial identity \cite[(8.30)]{DFI} is missing a factor $1/4$ on the right hand side.
By \cite[Lemma~8.2]{DFI},
\begin{align}
\label{Fiplus}\Phi_2^{+1}(s,ir)&=\left(s-\frac{1}{2}\right)\Gamma\left(\frac{s+ir}{2}\right)\Gamma\left(\frac{s-ir}{2}\right),\\
\label{Fiminus}\Phi_2^{-1}(s,ir)&=2\Gamma\left(\frac{s+ir+1}{2}\right)\Gamma\left(\frac{s-ir+1}{2}\right).
\end{align}

We shall derive the analytic properties of $D(g,x,s)$ from those of $F(g,x,s)$. Splitting the integral in \eqref{intrep} and applying \eqref{modular} in the form
\[\gser\left(\frac{a}{c}+\frac{iy}{c}\right)=-\gser\left(-\frac{\ov a}{c}+\frac{i}{cy}\right)
-\gconst\left(\frac{1}{cy}\right)-\gconst\left(\frac{y}{c}\right),\qquad 1>y>0,\]
\eqref{intrep} becomes
\begin{align}
\label{Fdecomp}\left(\frac{c}{\pi}\right)^s F\left(g,\frac{a}{c},s\right)
=&\sqrt{c}\int_1^\infty\left\{\gser\left(\frac{a}{c}+\frac{iy}{c}\right)y^{s-\frac{1}{2}}
-\gser\left(-\frac{\ov a}{c}+\frac{iy}{c}\right)y^{\frac{1}{2}-s}\right\}\frac{dy}{y}
+P(ir,c,s),\\
\intertext{where}
\label{Pdef}P(ir,c,s):=
&-\sqrt{c}\int_1^\infty\left\{\gconst\left(\frac{1}{cy}\right)+\gconst\left(\frac{y}{c}\right)\right\}
y^{\frac{1}{2}-s}\frac{dy}{y}.
\end{align}
By \eqref{gserbound} the first integral in \eqref{Fdecomp} is absolutely convergent for any $s\in\CC$, hence
it defines an entire function. Moreover, it becomes its own \emph{negative} under the substitution
$\frac{a}{c}\to-\frac{\ov a}{c}$ and $s\to 1-s$ : we shall say it is \emph{symmetric}
for short. The second integral \eqref{Pdef} can be calculated explicitly using \eqref{gconst},
\begin{equation}\label{Ppolar}
P(ir,c,s)=\sum_{\pm}c^{\mp ir}\left(\tfrac{1}{2}\pm ir\right)\theta\left(\tfrac{1}{2}\pm ir\right)
\left(\frac{1}{s\pm ir}-\frac{1}{1-s\pm ir}\right),\qquad \Re s>1.
\end{equation}
The function $P(ir,c,s)$ is meromorphic and symmetric in $s\in\CC$, hence
$\left(\frac{c}{\pi}\right)^s F\left(g,\frac{a}{c},s\right)$ is also symmetric and differs from
$P(ir,c,s)$ by an entire function. If we apply this conclusion for $-\frac{a}{c}$ in place of
$\frac{a}{c}$ and combine \eqref{Fdef} with $D^{\pm 1}(g,-x,s)=\pm D^{\pm 1}(g,x,s)$, then we see
that the following functions are entire and symmetric:
\begin{equation}\label{Dsep}
\left(\frac{c}{\pi}\right)^s\Phi_2^{+1}(s,ir)D^{+1}\left(g,\frac{a}{c},s\right)-P(ir,c,s)
\qquad\text{and}\qquad \left(\frac{c}{\pi}\right)^s\Phi_2^{-1}(s,ir)D^{-1}\left(g,\frac{a}{c},s\right).\end{equation}
Now we can argue exactly as on \cite[p.~597]{HM} (with $k=2$) to see that $D(g,x,s)$ satisfies the functional equation \eqref{Dfunct}, cf.\ \cite[(15)]{HM}. Here we note that the functions $\Psi_{k,it}^\pm$ on \cite[p.~597]{HM} should have been halved, therefore in \cite[(19)--(20)]{HM} the factors $2^{2s}$ are really $2^{2s-1}$.

We also need to analyze the poles of $D\left(g,\frac{a}{c},s\right)$ for which we go back to \eqref{Dsep}.
We see by \eqref{Fiminus} that $D^{-1}\left(g,\frac{a}{c},s\right)$ is entire.
By \eqref{Fiplus} and \eqref{Ppolar} the poles of $D^{+1}\left(g,\frac{a}{c},s\right)$ are at $s=1\pm ir$; they are simple with residues (cf.\ \eqref{kz})
\[\frac{1}{\Phi_2^{+1}(1\pm ir,ir)}\left(\frac{\pi}{c}\right)^{1\pm ir}c^{\mp ir}
\left(\tfrac{1}{2}\pm ir\right)\theta\left(\tfrac{1}{2}\pm ir\right)=\frac{\zeta(1\pm 2ir)}{c^{1\pm 2ir}}.\]
We note that the potential poles at $s=\pm ir$ coming from the poles of $P(ir,c,s)$ are
canceled by the poles of $\Phi_2^{+1}(s,ir)$ there, while the potential pole at $s=1/2$ coming from the zero of $\Phi_2^{+1}(s,ir)$ is canceled by the zero of $P(ir,c,s)$ there. This concludes our proof of Lemma~\ref{Dlemma}.


\begin{thebibliography}{CFKRS}

\bibitem[Bl]{Bl} V. Blomer, \emph{Non-vanishing of class group $L$-functions at the central point}, Annales de l'Institut Fourier \textbf{54} (2004), 831--847.

\bibitem[BH]{BH} V. Blomer, G. Harcos, \emph{Hybrid bounds for twisted $L$-functions}, J. Reine Angew. Math. \textbf{621} (2008), 53--79.

\bibitem[BHM]{BHM} V. Blomer, G. Harcos, P. Michel, \emph{Bounds for modular $L$-functions in the level aspect} Ann. Scient. Ec. Norm. Sup. \textbf{40} (2007), 697--740

\bibitem[CFKRS]{C+} J. Conrey, D. Farmer, J. Keating, M. Rubinstein, N. Snaith, \emph{Integral moments of $L$-functions}, Proc. Lond. Math. Soc. (3) \textbf{91} (2005), 33--104.

\bibitem[CH]{CH} H. Cohen, J. Oesterl\'e, \emph{Dimensions des espaces de formes modulaires}, In: Modular functions of one variable VI (Proc. Second Internat. Conf., Univ. Bonn, Bonn, 1976), Lecture Notes in Math. \textbf{627}, Springer, Berlin, 1977, pp. 69--78.

\bibitem[Du]{Du} W. Duke, \emph{Fourth moments of $L$-functions attached to newforms}, Comm. Pure Appl. Math. \textbf{41} (1988), 815--831.

\bibitem[DFI]{DFI} W. Duke, J. Friedlander, H. Iwaniec, \emph{The subconvexity problem for Artin $L$-functions}, Invent. Math. \textbf{149} (2002), 489--577.

\bibitem[DI]{DI} W. Duke, H. Iwaniec, \emph{Bilinear forms in the Fourier coefficients of half-integral weight cusp forms and sums over primes}, Math. Ann. \textbf{286} (1990), 783--802.

\bibitem[Er]{Er} A. Erd\'elyi et al., \emph{Tables of integral transforms, Vol. I. [based on notes left by H. Bateman]}, McGraw-Hill, New York, 1954.

\bibitem[Es]{Es} T. Estermann, \emph{On the representation of a number as the sum of two products}, Proc. Lond. Math. Soc. (2) \textbf{31} (1930), 123--133.

\bibitem[GR]{GR} I. S. Gradshteyn, I. M. Ryzhik, \emph{Tables of integrals, series, and products}, 7th edition, Academic Press, New York, 2007.

\bibitem[Ha]{Ha} G. Harcos, \emph{Uniform approximate functional equation for principal $L$-functions}, Int. Math. Res. Not. \textbf{2002}, 923--932.; {Erratum}, ibid. \textbf{2004}, 659--660.

\bibitem[HM]{HM} G. Harcos, P. Michel, \emph{The subconvexity problem for Rankin--Selberg $L$-functions and equidistribution of Heegner points. II}, Invent. Math. \textbf{163} (2006), 581--655.

\bibitem[HB]{HB} D. R. Heath-Brown, \emph{The fourth power moment of the Riemann zeta function}, Proc. Lond. Math. Soc. (3) \textbf{38} (1979), 385--422.

\bibitem[HL]{HL} J. Hoffstein, P. Lockhart, \emph{Coefficients of Maass forms and the Siegel zero (with an appendix by D. Goldfeld, J. Hoffstein and D. Lieman)}, Ann. of Math. \textbf{140} (1994), 161--181.

\bibitem[Iw1]{Iw2} H. Iwaniec, \emph{Small eigenvalues of Laplacian for $\Gamma_0(N)$}, Acta Arith. \textbf{56} (1990), 65--82.

\bibitem[Iw2]{Iw} H. Iwaniec, \emph{Topics in classical automorphic forms}, Graduate Studies in Mathematics \textbf{17}, Amer. Math. Soc., Providence, RI, 1997.

\bibitem[Iw3]{Iw3} H. Iwaniec, \emph{Spectral methods of automorphic forms}, 2nd edition, Graduate Studies in Mathematics 53, American Mathematical Society, Providence, RI; Revista Matem\'atica Iberoamericana, Madrid, 2002.

\bibitem[IK]{IK} H. Iwaniec, E. Kowalski, \emph{Analytic number theory}, American Mathematical Society Colloquium Publications \textbf{53}, Amer. Math. Soc., Providence, RI, 2004.

\bibitem[Ju1]{Ju} M. Jutila, \emph{On exponential sums involving the divisor function}, J. Reine Angew. Math. \textbf{355} (1985), 173--190.

\bibitem[Ju2]{Ju2} M. Jutila, \emph{A method in the theory of exponential sums}, Tata Lect. Notes Math. \textbf{80}, Bombay, 1987.

\bibitem[JM]{JM} M. Jutila, Y. Motohashi, \emph{Uniform bound for Hecke $L$-functions}, Acta Math. \textbf{195} (2005), 61--115.

\bibitem[Kh]{Kh} R. Khan, \emph{Non-vanishing of the symmetric square $L$-function at the central point}, Proc. Lond. Math. Soc. (3) \textbf{100} (2010), 736--762.

\bibitem[KZ]{KZ} H. H. Kim, Y. Zhang,  \emph{Divisor function for quaternion algebras and application to fourth moments of $L$-functions}, J. Number Theory \textbf{129} (2009), 3000--3019.

\bibitem[LLY]{LLY} Y.-K.Lau, J. Liu and Y. Ye, \emph{A new bound $k^{2/3+\eps}$ for Rankin--Selberg $L$-functions for Hecke congruence subgroups}, Int. Math. Res. Pap. \textbf{2006},  Art. ID 35090, 78 pp.

\bibitem[Li]{Li} W. Li, \emph{$L$-series of Rankin type and their functional equation}, Math. Ann. \textbf{244} (1979), 135--166.

\bibitem[LRS]{LRS} W. Luo, Z. Rudnick, and P. Sarnak, \emph{On the generalized Ramanujan conjecture for ${\rm GL}(n)$}, In: Automorphic forms, automorphic representations, and arithmetic, Proc. Sympos. Pure Math. \textbf{66}, Part 2, Amer. Math. Soc., Providence, RI, 1999, pp. 301--310.

\bibitem[Me]{Me} T. Meurman, \emph{On exponential sums involving the Fourier coefficients of Maass wave forms}, J. Reine Angew. Math. \textbf{384} (1988), 192--207.

\bibitem[Mol]{Mo} G. Molteni, \emph{Upper and lower bounds at $s=1$ for certain Dirichlet series with
Euler product}, Duke Math. J. \textbf{111} (2002), 133--158.

\bibitem[Mot]{Mot} Y. Motohashi, \emph{Spectral theory of the Riemann zeta-function}, Cambridge Tracts in Mathematics \textbf{127}, Cambridge University Press, Cambridge, 1997.

\bibitem[PS]{PS} Y. Petridis, P. Sarnak, \emph{Quantum unique ergodicity for $\SL_2(\mathcal{O})\backslash \mathbf{H}^3$ and estimates for $L$-functions}, J. Evol. Equations \textbf{1} (2001), 277--290.

\bibitem[Ra1]{Ra2} D. Ramakrishnan, \emph{Modularity of the Rankin--Selberg $L$-series, and multiplicity one for ${\rm SL}(2)$}, Ann. of Math. \textbf{152} (2000), 45--111.

\bibitem[Ra2]{Ra} D. Ramakrishnan, \emph{Irreducibility and cuspidality}, In: Representation theory and automorphic forms, Progr. Math. \textbf{255}, Birkh\"auser Boston, Boston, MA, 2008, 1--27.

\bibitem[Sa1]{Sa1} P. Sarnak, \emph{Fourth moments of Gr\"ossencharakteren zeta functions}, Comm. Pure Appl. Math. \textbf{38} (1985), 167--178.

\bibitem[Sa2]{Sa} P. Sarnak, \emph{Estimates for Rankin--Selberg $L$-functions and quantum unique ergodicity}, J. Funct. Anal. \textbf{184} (2001), 419--453.

\bibitem[Ti]{Ti} E. C. Titchmarsh, \emph{The theory of the Riemann zeta-function}, Second edition [edited and with a preface by D. R. Heath-Brown], Clarendon Press, Oxford University Press, New York, 1986.

\bibitem[Y1]{Y} M. P. Young, \emph{The second moment of $GL(3) \times GL(2)$ $L$-functions, integrated}, preprint available at {\tt arXiv:0903.1575}

\bibitem[Y2]{Y1} M. P. Young, \emph{The second moment of $GL(3) \times GL(2)$ $L$-functions at special points}, preprint available at {\tt arXiv:0903.1579}

\end{thebibliography}
\end{document}